%% file: non-smooth-median-3-Xiaohan.tex
\documentclass[final]{article}

\usepackage{nips_2018}

\usepackage[utf8]{inputenc} 
\usepackage[T1]{fontenc}    
\usepackage{hyperref}       
\usepackage{url}            
\usepackage{booktabs}       
\usepackage{amsfonts}       
\usepackage{nicefrac}       
\usepackage{microtype}      %

\usepackage{color}
\usepackage{subcaption}

\let\OLDthebibliography\thebibliography
\renewcommand\thebibliography[1]{
  \OLDthebibliography{#1}
  \setlength{\parskip}{0pt}
  \setlength{\itemsep}{2pt plus 0.3ex}
}

\usepackage{graphicx}				
\usepackage{amssymb,amsmath,amsthm}
\usepackage{enumitem}

\usepackage{algorithm}

\newcommand{\wt}[1]{\widetilde{#1}}
\newcommand{\mc}[1]{\mathcal{#1}}

\input{xiaohan}

\input{stas}

\title{Solving Non-smooth Constrained Programs with Lower Complexity than $\mathcal{O}(1/\varepsilon)$: A Primal-Dual Homotopy Smoothing Approach}

\author{
  Xiaohan Wei \\
  Department of Electrical Engineering \\
  University of Southern California\\
  Los Angeles, CA, USA, 90089 \\
  \texttt{xiaohanw@usc.edu} \\
  \And
  Hao Yu \\
  Alibaba Group (U.S.) Inc.\\
  Bellevue, WA, USA, 98004\\
  \texttt{hao.yu@alibaba-inc.com} \\
  \And
  Qing Ling \\
  School of Data and Computer Science \\
  Sun Yat-Sen University\\
  Guangzhou, China, 510006 \\
  \texttt{lingqing556@mail.sysu.edu.cn} \\
  \And
  Michael J. Neely \\
  Department of Electrical Engineering \\
  University of Southern California\\
  Los Angeles, CA, USA, 90089 \\
  \texttt{mikejneely@gmail.com } \\
}

\begin{document}

\maketitle

\begin{abstract}
We propose a new primal-dual homotopy smoothing algorithm for a linearly constrained convex program, where neither the primal nor the dual function has to be smooth or strongly convex. The best known iteration complexity solving such a non-smooth problem is $\mathcal{O}(\varepsilon^{-1})$. In this paper, 
we show that by leveraging a local error bound condition on the dual function, the proposed algorithm can achieve a better primal convergence time of  $\mathcal{O}\l(\varepsilon^{-2/(2+\beta)}\log_2(\varepsilon^{-1})\r)$, where $\beta\in(0,1]$ is a local error bound parameter. 
As an example application of the general algorithm, we show that the distributed geometric median problem, which can be formulated as a constrained convex program, has its dual function non-smooth but satisfying the aforementioned local error bound condition with $\beta=1/2$, therefore enjoying a convergence time of $\mathcal{O}\l(\varepsilon^{-4/5}\log_2(\varepsilon^{-1})\r)$. This result improves upon the $\mathcal{O}(\varepsilon^{-1})$ convergence time bound achieved by existing distributed optimization algorithms. Simulation experiments also demonstrate the performance of our proposed algorithm. 
\end{abstract}

\section{Introduction}
We consider the following linearly constrained convex optimization problem:
\begin{align}
\min&~~f(\mf x)     \label{prob-1}\\
\text{s.t.}&~~\mf A \mf x - \mf b = 0,~~\mf x \in \mathcal{X},    \label{prob-2}
\end{align}
where $\mathcal{X}\subseteq \mb R^d$ is a compact convex set,  $f: \mb R^d \rightarrow \mb R$ is a convex function, $\mf A \in\mb R^{N\times d},~\mf b\in\mb R^N$. Such an optimization problem has been studied in numerous works under various application scenarios such as machine learning (\cite{yurtsever2015universal}), signal processing (\cite{ling2010decentralized}) and communication networks (\cite{yu2017new}). The goal of this work is to design new algorithms for (\ref{prob-1}-\ref{prob-2}) achieving an 
$\varepsilon$ approximation with better convergence time than $\mathcal{O}(1/\varepsilon)$.


\subsection{Optimization algorithms related to constrained convex program}
Since enforcing the constraint $\mf A \mf x - \mf b = 0$ generally requires a significant amount of computation in large scale systems, the majority of the scalable algorithms solving problem (\ref{prob-1}-\ref{prob-2}) are of primal-dual type. Generally, the efficiency of these algorithms depends on two key properties of the dual function of (\ref{prob-1}-\ref{prob-2}), namely, the Lipschitz gradient and strong convexity. When the dual function of (\ref{prob-1}-\ref{prob-2}) is smooth, primal-dual type algorithms with Nesterov's acceleration on the dual of \eqref{prob-1}-\eqref{prob-2} can achieve a convergence time of $\mathcal{O}(1/\sqrt{\varepsilon})$ (e.g. \cite{yurtsever2015universal,tran2018smooth})\footnote{Our convergence time to achieve within $\varepsilon$ of optimality is in terms of \emph{number of  (unconstrained) maximization steps} 
$\arg\max_{\mathbf x \in \mathcal{X}} [\lambda^T (\mathbf{Ax}-\mathbf{b}) - f(\mathbf{x}) - \frac{\mu}{2}\|\mathbf{x}-\tilde{\mathbf{x}}\|^2 ] $
where constants $\lambda, A, \tilde{x}, \mu$ are known.  This is a standard measure of convergence time for Lagrangian-type algorithms that turn a constrained problem into a sequence of unconstrained problems.}. When the dual function has both the Lipschitz continuous gradient and the strongly convex property, algorithms such as dual subgradient and ADMM enjoy a linear convergence $\mathcal{O}(\log (1/\varepsilon))$ (e.g. \cite{yu2017convergence,deng2016global}). 
However, when neither of the properties is assumed, the basic dual-subgradient type algorithm gives a relatively worse $\mathcal{O}(1/\varepsilon^2)$ convergence time (e.g. \cite{wei2015probabilistic, wei2018primal}), while its improved variants 
 yield a convergence time of $\mathcal{O}(1/\varepsilon)$ (e.g. \cite{lan2013iteration, deng2017parallel, yu2017simple, yurtsever2018conditional, gidel2018frank}).

More recently, several works seek to achieve a better convergence time than $\mathcal{O}(1/\varepsilon)$ under weaker assumptions than Lipschitz gradient and strong convexity of the dual function. Specifically,
building upon the recent progress on the gradient type methods for optimization with H$\ddot{o}$lder continuous gradient (e.g. \cite{nesterov2015complexity, nesterov2015universal}), the work \cite{yurtsever2015universal} develops a primal-dual gradient method solving (\ref{prob-1}-\ref{prob-2}), which achieves a convergence time of $\mathcal{O}(1/\varepsilon^{\frac{1+\nu}{1+3\nu}})$, where $\nu$ is the modulus of H$\ddot{o}$lder continuity on the gradient of the dual function of the formulation (\ref{prob-1}-\ref{prob-2}).\footnote{The gradient of function $g(\cdot)$ is H$\ddot{o}$lder continuous with modulus $\nu\in(0,1]$ on a set $\mathcal{X}$ if $\|\nabla g(\mf x) - \nabla g(\mf y)\|\leq L_\nu\|\mf x - \mf y\|^{\nu},~\forall \mf x,\mf y \in\mathcal{X}$, where $\|\cdot\|$ is the vector 2-norm and $L_\nu$ is a constant depending on $\nu$.} On the other hand, the work \cite{yu2017convergence} shows that when the dual function has Lipschitz continuous gradient and satisfies a locally quadratic property (i.e. a local error bound with $\beta =1/2$, see Definition \ref{def:local-error} for details), which is weaker than strong convexity, one can still obtain a linear convergence with a dual subgradient algorithm. A similar result has also been proved for ADMM in \cite{han2015linear}.

In the current work, we aim to address the following question: \textit{Can one design a scalable algorithm with lower complexity than $\mathcal{O}(1/\varepsilon)$ solving (\ref{prob-1}-\ref{prob-2}), when both the primal and the dual functions are possibly non-smooth?} More specifically, we look at a class of problems with dual functions satisfying only a local error bound, and show that indeed one is able to obtain a faster primal convergence via a \textit{primal-dual homotopy smoothing method}  under a local error bound condition on the dual function. 

Homotopy methods were first developed in the statistics literature in relation to the model selection problem for LASSO, where, instead of computing a single solution for LASSO, one computes a complete solution path by varying the regularization parameter from large to small (e.g. \cite{osborne2000new, xiao2013proximal}).\footnote{ The word ``homotopy'', which was adopted in \cite{osborne2000new}, refers to the fact that the mapping from regularization parameters to the set of solutions of the LASSO problem is a continuous piece-wise linear function.} On the other hand, the smoothing technique for minimizing a non-smooth convex function of the following form was first considered in \cite{nesterov2005smooth}:
\begin{equation}\label{eq:nesterov}
\Psi(\mf x) =  g(\mf x) + h(\mf x), ~\mf x\in\Omega_1
\end{equation}
where $\Omega_1\subseteq\mb R^d$ is a closed convex set, $h(\mf x)$ is a convex smooth function, and $g(\mf x)$ can be explicitly written as 
\begin{equation}\label{eq:dual-form}
g(\mf x) = \max_{\mf u\in\Omega_2} \dotp{\mf A\mf x}{\mf u} - \phi(\mf u),
\end{equation}
where for any two vectors $\mf a, \mf b\in\mb R^d$, $\dotp{\mf a}{\mf b} = \mf a^T\mf b$, $\Omega_1\subseteq\mb R^d$ is a closed convex set, and $\phi(\mf u)$ is a convex function. By adding a strongly concave proximal function of $\mf u$ with a smoothing parameter $\mu>0$ into the definition of $g(\mf x)$, one can obtain a smoothed approximation of $\Psi(\mf x)$ with smooth modulus $\mu$. Then, \cite{nesterov2005smooth} employs the accelerated gradient method on the smoothed approximation (which delivers a $\mathcal{O}(1/\sqrt{\varepsilon})$ convergence time for the approximation), and sets the parameter to be $\mu = \mathcal{O}(\varepsilon)$, which gives an overall convergence time of $\mathcal{O}(1/\varepsilon)$. 
An important follow-up question is that whether or not such a smoothing technique can also be applied to solve (\ref{prob-1}-\ref{prob-2}) with the same  primal convergence time. This question is answered in subsequent works \cite{necoara2008application, li2016fast, tran2018smooth}, where they show that indeed one can also obtain an 
$\mathcal{O}(1/\varepsilon)$ primal convergence time for the problem (\ref{prob-1}-\ref{prob-2}) via smoothing. 


Combining the homotopy method with a smoothing technique to solve problems of the form \eqref{eq:nesterov} has been considered by a series of works including \cite{yang2015rsg}, \cite{xu2016homotopy} and \cite{xu2017admm}. Specifically, the works \cite{yang2015rsg} and \cite{xu2016homotopy} consider a multi-stage algorithm which starts from a large smoothing parameter $\mu$ and then decreases this parameter over time. They show that when the function $\Psi(\mf x)$ satisfies a local error bound with parameter $\beta\in(0,1]$, such a combination gives an improved convergence time of $\mathcal{O}(\log(1/\varepsilon) / \varepsilon^{1-\beta})$ minimizing the unconstrained problem \eqref{eq:nesterov}. The work \cite{xu2017admm} shows that the homotopy method can also be combined with ADMM to achieve a faster convergence solving problems of the form 
\[
\min_{\mf x\in\Omega_1} f(\mf x) + \psi(\mf A\mf x - \mf b),
\]
where $\Omega_1$ is a closed convex set, 
$f,\psi$ are both convex functions with $f(\mf x) + \psi(\mf A\mf x - \mf b)$ satisfying the local error bound, and the proximal operator of $\psi(\cdot)$ can be easily computed. However, due to the restrictions on the function $\psi$ in the paper, it \textit{cannot} be extended to handle problems of the form (\ref{prob-1}-\ref{prob-2}).\footnote{The result in \cite{xu2017admm} heavily depends on the assumption that the subgradient of $\psi(\cdot)$ is defined everywhere over the set $\Omega_1$ and uniformly bound by some constant $\rho$, which excludes the choice of indicator functions necessary to deal with constraints in the ADMM framework.}

\textbf{Contributions:} In the current work, we show a multi-stage homotopy smoothing method enjoys a primal convergence time $\mathcal{O}\l(\varepsilon^{-2/(2+\beta)}\log_2(\varepsilon^{-1})\r)$ solving (\ref{prob-1}-\ref{prob-2}) when the dual function satisfies a local error bound condition with $\beta\in(0,1]$. Our convergence time to achieve within 
$\varepsilon$ of optimality is in terms of \emph{number of  (unconstrained) maximization steps} 
$\arg\max_{x \in \mathcal{X}} [\lambda^T (\mathbf{Ax}-\mathbf{b}) - f(\mathbf{x}) - \frac{\mu}{2}||\mathbf{x}-\widetilde{\mathbf{x}}||^2 ]$,
where constants $\lambda, \mathbf{A}, \widetilde{\mathbf{x}}, \mu$ are known, which is a standard measure of convergence time for Lagrangian-type algorithms that turn a constrained problem into a sequence of unconstrained problems.
The algorithm essentially restarts a weighted primal averaging process at each stage using the last Lagrange multiplier computed.
 This result improves upon the earlier $\mathcal{O}(1/\varepsilon)$ result by (\cite{necoara2008application, li2016fast}) and at the same time extends the scope of homotopy smoothing method to solve a new class of problems involving constraints (\ref{prob-1}-\ref{prob-2}). It is worth mentioning that a similar restarted smoothing strategy is proposed in a recent work \cite{tran2018smooth} to solve problems including (\ref{prob-1}-\ref{prob-2}), where they show that, empirically, restarting the algorithm from the Lagrange multiplier computed from the last stage improves the convergence time. Here, we give one theoretical justification of such an improvement.

\subsection{The distributed geometric median problem}


The geometric median problem, also known as  the Fermat-Weber problem, has a long history (e.g. see \cite{weiszfeld2009point} for more details). Given a set of $n$ points $\mf b_1,~\mf b_2,~\cdots,~\mf b_n\in\mb{R}^d$, we aim to find one point $\mf{x}^*\in\mb R^d$ so as to minimize the sum of the Euclidean distance, i.e.
\begin{equation}\label{eq:geo-median}
\mf x^* \in \argmin_{\mf x\in\mb R^d}\sum_{i=1}^n\|\mf x - \mf b_i\|,
\end{equation}
which is a \textit{non-smooth} convex optimization problem.
It can be shown that the solution to this problem is \textit{unique} as long as $\mf b_1,~\mf b_2,~\cdots,~\mf b_n\in\mb{R}^d$ are not co-linear.
Linear convergence time algorithms solving \eqref{eq:geo-median} have also been developed in several works (e.g. \cite{xue1997efficient}, \cite{parrilo2003minimizing}, \cite{cohen2016geometric}). Our motivation of studying this problem is driven by its recent application in 
distributed statistical estimation, in which data are assumed to be randomly spreaded to multiple connected computational agents that produce intermediate estimators, and then, these intermediate estimators are aggregated in order to compute some statistics of the whole data set. Arguably one of the most widely used aggregation procedures is computing the \textit{geometric median} of the local estimators (see, for example,  \cite{duchi2014optimality}, \cite{minsker2014robust}, \cite{minsker2017distributed}, \cite{yin2018byzantine}).
It can be shown that the geometric median is robust against arbitrary corruptions of local estimators in the sense that the final estimator is stable as long as at least half of the nodes in the system perform as expected.

\textbf{Contributions:}
As an example application of our general algorithm, we look at the problem of computing the solution to \eqref{eq:geo-median} in a distributed scenario over a network of $n$ agents without any central controller, where each agent holds a local vector   
$\mathbf{b}_i$. Remarkably, we show theoretically that such a problem, when formulated as (\ref{prob-1}-\ref{prob-2}), has its dual function \textit{non-smooth but locally quadratic}. Therefore, applying our proposed primal-dual homotopy smoothing method gives a convergence time of $\mathcal{O}\l(\varepsilon^{-4/5}\log_2(\varepsilon^{-1})\r)$. This result improves upon the performance bounds of the previously known decentralized optimization algorithms (e.g. PG-EXTRA \cite{shi2015proximal} and decentralized ADMM \cite{shi2014linear}), which do not take into account the special structure of the problem and only obtain a convergence time of $\mathcal{O}\l(1/\varepsilon\r)$. Simulation experiments also demonstrate the superior ergodic convergence time of our algorithm compared to other algorithms.


\section{Primal-dual Homotopy Smoothing}
\subsection{Preliminaries}
The Lagrange dual function of (\ref{prob-1}-\ref{prob-2}) is defined as follows:\footnote{Usually, the Lagrange dual is defined as $\min_{\mf x\in\mc X}~ \dotp{\lambda}{ \mf A \mf x - \mf b} + f(\mf x)$. Here, we flip the sign and take the maximum for no reason other than being consistent with the form \eqref{eq:dual-form}.}
\begin{equation}\label{eq:dual-function}
F(\lambda) := \max_{\mf x\in\mc X} ~~\l\{- \dotp{\lambda}{ \mf A \mf x - \mf b} - f(\mf x)\r\},
\end{equation}
where $\lambda\in\mb R^N$ is the dual variable, $\mc X$ is a compact convex set and the minimum of the dual function is
$
F^* := \min_{\lambda\in\mb R^N} F(\lambda).
$
For any closed set $\mc K\subseteq\mb R^d$ and $\mf x\in\mb R^d$, define the distance function of $\mf x$ to the set $\mc K$ as
\[
\text{dist}(\mf x, \mc K) := \min_{\mf y\in\mc K}\|\mf x- \mf y\|,
\]
where $\|\mf x\|:= \sqrt{\sum_{i=1}^dx_i^2}$. For a convex function $F(\lambda)$, the $\delta$-sublevel set $\mc S_\delta$ is defined as
\begin{equation}\label{eq:sublevel}
\mc S_\delta := \{\lambda\in\mb R^N:~F(\lambda) - F^*\leq \delta\}.
\end{equation}
Furthermore, for any matrix $\mf{A}\in\mb R^{N\times d}$, we use $\sigma_{\max}(\mf A^T\mf A)$ to denote the largest eigenvalue of $\mf A^T\mf A$. Let 
\begin{equation}\label{eq:optimal-set-1}
\Lambda^*:= \l\{\lambda^*\in\mb R^N:~F(\lambda^*)\leq F(\lambda),~\forall \lambda\in\mb R^N\r\}
\end{equation}
 be the set of optimal Lagrange multipliers. Note that if 
the constraint $\mf A\mf x=\mf b$ is feasible,  then 
$\lambda^*\in \Lambda^*$ implies  $\lambda^* + \mf v\in \Lambda^*$ for any $\mf v$ that satisfies $\mf A^T\mf v = 0$. 
 The following definition introduces the notion of local error bound.
\begin{definition}\label{def:local-error}
Let $F(\lambda)$ be a convex function over $\lambda\in\mb R^N$.
Suppose $\Lambda^*$ is non-empty. The function $F(\lambda)$ is said to satisfy the local error bound with parameter $\beta\in(0,1]$ if $\exists \delta >0$ such that for any $\lambda\in\mc S_\delta$,
\begin{equation}\label{eq:local-error-bound}
\text{dist}(\lambda,\Lambda^*)\leq C_\delta (F(\lambda) - F^*)^\beta,
\end{equation}
where $C_\delta$ is a positive constant possibly depending on $\delta$. In particular, when $\beta = 1/2$, $F(\lambda)$ is said to be locally quadratic and when $\beta = 1$, it is said to be locally linear.
\end{definition}
\begin{remark}
Indeed, a wide range of popular optimization problems satisfy the local error bound condition.
The work \cite{tseng2010approximation} shows that if $\mathcal{X}$ is a polyhedron,  $f(\cdot)$ has Lipschitz continuous gradient and is strongly convex, then the dual function of (\ref{prob-1}-\ref{prob-2}) is locally linear. The work \cite{burke1996unified} shows that when the objective is linear and $\mathcal{X}$ is a convex cone, the dual function is also locally linear. The values of $\beta$ have also been computed for several other problems (e.g. \cite{Pang1997, yang2015rsg}).
\end{remark}

\begin{definition}
Given an accuracy level $\varepsilon >0$, a vector $\mf x_0\in\mc X$ is said to achieve an $\varepsilon$ approximate solution regarding problem (\ref{prob-1}-\ref{prob-2}) if
\[
f(\mf x_0) - f^*\leq \mathcal{O}(\varepsilon),~~\|\mf A \mf x_0 - \mf b\|\leq \mc{O}(\varepsilon),
\]
where $f^*$ is the optimal primal objective of (\ref{prob-1}-\ref{prob-2}).
\end{definition}


Throughout the paper, we adopt the following assumptions:
\begin{assumption}\label{assumption}~
(a) The feasible set $\{\mf x \in \mathcal{X}:~\mf A \mf x - \mf b = 0\}$ is nonempty and non-singleton.\\
(b) The set $\mc X$ is bounded, i.e.
$
\sup_{\mf x,\mf y\in\mc X}\|\mf x - \mf y\|\leq D,
$
for some positive constant $D$. Furthermore, the function $f(\mf x)$ is also bounded, i.e. 
$
\max_{\mf x\in \mc X}|f(\mf x)|\leq M,
$
for some positive constant $M$.\\
(c) The dual function defined in \eqref{eq:dual-function} satisfies the local error bound for some parameter $\beta \in(0,1]$ and some level $\delta>0$.\\
(d) Let $\mc P_{\mf A}$ be the projection operator onto the column space of $\mf A$. 
There exists a unique vector $\nu^*\in\mb R^N$ such that for any $\lambda^*\in\Lambda^*$, $\mc P_{\mf A}\lambda^* = \nu^*$, i.e. 
$\Lambda^* = \l\{ \lambda^*\in\mb R^N:  \mc P_{\mf A}\lambda^* = \nu^*\r\}$.
\end{assumption}
Note that assumption (a) and (b) are  very mild and quite standard. For most applications, it is enough to check (c) and (d).
We will show, for example, in Section \ref{sec:geo-median} that the distributed geometric median problem satisfies all the assumptions. 
Finally, we say a function $g:\mc X\rightarrow \mb R$ is smooth with modulus $L>0$ if
\[
\|\nabla g(\mf x)  - \nabla g(\mf y)\| \leq L\|\mf x - \mf y\|,~\forall \mf x, \mf y\in\mc X.
\]

\subsection{Primal-dual homotopy smoothing algorithm}
This section introduces our proposed algorithm for optimization problem (\ref{prob-1}-\ref{prob-2}) satisfying Assumption \ref{assumption}. The idea of smoothing is to introduce a smoothed Lagrange dual function $F_\mu(\lambda)$ that approximates the original possibly non-smooth dual function 
$F(\lambda)$ defined in \eqref{eq:dual-function}. 

For any constant $\mu > 0$,
define 
\begin{equation}\label{eq:primal-mu}
f_{\mu}(\mf x) = f(\mf x) + \frac\mu2\|\mf x - \wt{\mf x}\|^2,
\end{equation}
where $\wt{\mf x}$ is an arbitrary fixed point in $\mc X$. For simplicity of notation, we drop the dependency on $\wt{\mf x}$ in the definition of $f_{\mu}(\mf x)$.
Then, by the boundedness assumption of $\mc X$, we have 
$
f(\mf x)\leq f_{\mu}(\mf x) \leq f(\mf x) + \frac\mu2D^2,~~\forall \mf x\in\mc X.
$
For any $\lambda\in\mb R^N$, define
\begin{equation}\label{eq:dual-function-2}
F_\mu(\lambda) = \max_{\mf x\in\mc X} - \dotp{\lambda}{\mf A \mf x - \mf b}  - f_{\mu}(\mf x)
\end{equation}
as the smoothed dual function. The fact that $F_\mu(\lambda)$ is indeed smooth with modulus $\mu$ follows from Lemma \ref{lem:smooth} in the Supplement.
Thus, one is able to apply an accelerated gradient descent algorithm on this modified Lagrange dual function, which is detailed in Algorithm 1 below, starting from an initial primal-dual pair $(\wt{\mf x},\wt{\mf\lambda})\in\mb R^d\times\mb R^N$.

\begin{algorithm}\label{alg} 
\caption{Primal-Dual Smoothing: PDS$\l(\widetilde{\lambda},\wt{\mf x}, \mu,T\r)$}
\label{alg:new-alg}
Let $\lambda_0 = \lambda_{-1} = \widetilde{\lambda}$ and $\theta_0 = \theta_{-1} = 1$. \\
\textbf{For} $t = 0$ to $T-1$ \textbf{do}
\begin{itemize}
\item Compute a tentative dual multiplier:
$
\widehat{\lambda}_t = \lambda_t + \theta_t(\theta_{t-1}^{-1} - 1)(\lambda_t - \lambda_{t-1}),
$
\item Compute the primal update:
$
\mf x(\widehat \lambda_t) = \text{argmax}_{\mf x\in\mc X}~ - \dotp{\widehat\lambda_t}{\mf A \mf x - \mf b} - f(\mf x) - \frac\mu2\|\mf x - \wt{\mf x}\|^2.
$
\item Compute the dual update:
$
\lambda_{t+1} = \widehat \lambda_t + \mu(\mf A \mf x(\widehat \lambda_t) - \mf b). 
$
\item Update the stepsize: $\theta_{t+1} = \frac{\sqrt{\theta_t^4 + 4\theta_t^2} - \theta_t^2}{2}$.
\end{itemize}
\textbf{end for}\\
\textbf{Output:} $\overline{\mf x}_T = \frac{1}{S_T}\sum_{t= 0}^{T-1}\frac{1}{\theta_t}\mf x(\widehat \lambda_t)$ and $\lambda_T$, where 
$S_T=\sum_{t= 0}^{T-1}\frac{1}{\theta_t}$.
\end{algorithm}
Our proposed algorithm runs Algorithm 1 in multiple stages, which is detailed in Algorithm 2 below.

\begin{algorithm}\label{alg-2} 
\caption{Homotopy Method:}
\label{alg:new-alg-2}
Let $\varepsilon_0$ be a fixed constant and $\varepsilon < \varepsilon_0$ be the desired accuracy. 
Set $\mu_0 = \frac{\varepsilon_0}{D^2}$, $\lambda^{(0)} = 0$, $\overline{\mf x}^{(0)}\in\mc X$, the number of stages $K\geq \lceil \log_2(\varepsilon_0/\varepsilon) \rceil + 1$, and the time horizon during each stage $T\geq1$.\\
\textbf{For} $k = 1$ to $K$ \textbf{do}
\begin{itemize}
\item Let $\mu_{k} = \mu_{k-1}/2$.
\item Run the primal-dual smoothing algorithm ($\lambda^{(k)}$, $\overline{\mf x}^{(k)}$) = PDS$\l(\lambda^{(k-1)},\overline{\mf x}^{(k-1)},\mu_k,T\r)$.
\end{itemize}
\textbf{end for}\\
\textbf{Output:} $\overline{\mf x}^{(K)}$.
\end{algorithm}

\section{Convergence Time Results}
We start by defining the set of optimal Lagrange multipliers for the smoothed problem:\footnote{By Assumption \ref{assumption}(a) and Farkas' Lemma, this is non-empty.}
\begin{equation}\label{eq:optimal-set-2}
\Lambda_\mu^*:= \l\{ \lambda_\mu^*\in\mb R^N: F_\mu(\lambda_\mu^*)\leq F_\mu(\lambda),~\forall \lambda\in\mb R^N \r\}
\end{equation}
Our convergence time analysis involves two steps. The first step is to derive a primal convergence time bound for Algorithm 1, which involves the location information of the initial Lagrange multiplier at the beginning of this stage. The details are given in Supplement \ref{proof-each-stage}.
\begin{theorem}\label{main-thm-1}
Suppose Assumption \ref{assumption}(a)(b) holds. 
For any $T\geq1$ and any initial vector $(\wt{\mf x},\wt{\mf\lambda})\in\mb R^d\times\mb R^N$, we have the following performance bound regarding Algorithm 1, 
\begin{align}
&f\l( \overline{\mf x}_T \r) - f^*\leq \|\mc P_{\mf A}\wt\lambda^*\|\cdot\|\mf A\overline{\mf x}_T-\mf b\| 
+ \frac{\sigma_{\max}(\mf A^T\mf A)}{2\mu S_T}\l\|\wt\lambda^*- \wt{\lambda}\r\|^2 + \frac{\mu D^2}{2}, \label{eq:objective}\\
&\|\mf A\overline{\mf x}_T-\mf b\| \leq \frac{2\sigma_{\max}(\mf A^T\mf A)}{\mu S_T}\l(\l\|\wt\lambda^*- \wt{\lambda}\r\| + \text{dist}(\lambda_\mu^*,\Lambda^*)\r),
\label{eq:constraint}
\end{align}
where $\wt\lambda^* \in \text{argmin}_{\lambda^*\in\Lambda^*}\|\lambda^* - \wt\lambda\|$, 
$\overline{\mf x}_T:= \frac{1}{S_T}\sum_{t=0}^{T-1}\frac{\mf x(\widehat{\lambda}_t)}{\theta_t}$, $S_T = \sum_{t=0}^{T-1}\frac{1}{\theta_t}$ and 
$\lambda_\mu^*$ is any point in $\Lambda_\mu^*$ defined in \eqref{eq:optimal-set-2}.
\end{theorem}
An inductive argument shows that $\theta_t\leq 2/(t+2)~\forall t\geq0$. Thus, Theorem \ref{main-thm-1} already gives an $\mathcal{O}(1/\varepsilon)$ convergence time by setting $\mu = \varepsilon$ and $T=1/\varepsilon$. Note that this is the best trade-off we can get from Theorem \ref{main-thm-1} when simply bounding the terms $\|\wt\lambda^*- \wt{\lambda}\|$ and
$\text{dist}(\lambda_\mu^*,\Lambda^*)$ by constants. 
To see how this bound leads to an improved convergence time when running in multiple rounds, suppose 
the computation from the last round gives a 
$\wt\lambda$ that is close enough to the optimal set $\Lambda^*$, then, $\|\wt\lambda^*- \wt{\lambda}\|$ would be small. When the local error bound condition holds, one can show that $\text{dist}(\lambda_\mu^*,\Lambda^*)\leq \mathcal{O}(\mu^\beta)$.
As a consequence, one is able to choose $\mu$ smaller than $\varepsilon$ and get a better trade-off. Formally, we have the following overall performance bound. The proof is given in Supplement \ref{proof-overall}.
\begin{theorem}\label{main-thm-2}
Suppose Suppose Assumption \ref{assumption} holds,  $\varepsilon_0\geq\max\{ 2M,1\}$, $0<\varepsilon\leq \min\{\delta/2,2M,1\}$,  $T\geq\frac{2DC_\delta\sqrt{\sigma_{\max}(\mf A^T\mf A)}(2M)^{\beta/2}}{\varepsilon^{2/(2+\beta)}}$. 
The proposed homotopy method achieves the following objective and constraint violation bound:
\begin{align*}
&f(\overline{\mf x}^{(K)}) - f^*\leq \l(\frac{24\|\mc P_{\mf A}\lambda_*\|(1+C_\delta)}{C_\delta^2(2M)^{2\beta}} + \frac{6}{C_\delta^2(2M)^{2\beta}}+\frac14\r)\varepsilon,\\
&\|\mf A \overline{\mf x}^{(K)} - \mf b\|\leq \frac{24(1+C_\delta)}{C_\delta^2(2M)^{\beta}}\varepsilon,
\end{align*}
with running time $\frac{2DC_\delta\sqrt{\sigma_{\max}(\mf A^T\mf A)}(2M)^{\beta/2}}{\varepsilon^{2/(2+\beta)}}(\lceil \log_2(\varepsilon_0/\varepsilon) \rceil + 1)$, i.e. the algorithm achieves an $\varepsilon$ approximation with convergence time 
$\mathcal{O}\l(\varepsilon^{-2/(2+\beta)}\log_2(\varepsilon^{-1})\r)$.
\end{theorem}

\section{Distributed Geometric Median}\label{sec:geo-median}
Consider the problem of computing the geometric median over a connected network $(\mc V, \mc E)$, where $\mc V=\{1,2,\cdots,n\}$ is a set of $n$ nodes, $\mc E=\{e_{ij}\}_{i,j\in\mc V}$ is a collection of undirected edges, $e_{ij}=1$ if there exists an undirected edge between node $i$ and node $j$, and $e_{ij}=0$ otherwise. Furthermore, $e_{ii} = 1,~\forall i\in\{1,2,\cdots,n\}$.Furthermore, since the graph is undirected, we always have $e_{ij} = e_{ji},~\forall i,j\in\{1,2,\cdots,n\}$.
Two nodes $i$ and $j$ are said to be neighbors of each other if $e_{ij} = 1$. 
Each node $i$ holds a local vector 
$\mf b_i\in\mb R^d$, and the goal is to compute the solution to \eqref{eq:geo-median} without having a central controller, i.e. each node can only communicate with its neighbors. 

Computing geometric median over a network has been considered in several works previously and various distributed algorithms have been developed such as decentralized subgradient methd (DSM, \cite{nedic2009distributed,yuan2016convergence}), 
PG-EXTRA (\cite{shi2015proximal}) and ADMM (\cite{shi2014linear,deng2017parallel}). The best known convergence time for this problem is  $\mathcal{O}(1/\varepsilon)$. In this section, we will show that it can be written in the form of problem (\ref{prob-1}-\ref{prob-2}), has its Lagrange dual function \textit{locally quadratic} and optimal Lagrange multiplier unique up to the null space of $\mf A$, thereby satisfying Assumption \ref{assumption}.

Throughout this section, we assume that $n\geq3$, $\mf b_1,~\mf b_2,~\cdots,~\mf b_n\in\mb{R}^d$ are not co-linear and they are distinct (i.e. $\mf b_i \neq \mf b_j$ if $i\neq j$).
We start by defining a mixing matrix $\wt{\mf W}\in\mb R^{n\times n}$ with respect to this network. The mixing matrix will have the following properties:
\begin{enumerate}
\item Decentralization: The $(i,j)$-th entry 
$\wt w_{ij} = 0$ if $e_{ij} = 0$.
\item Symmetry: $\wt{\mf W} = \wt{\mf W}^T$.
\item The null space of $\mf I_{n\times n} -\wt{\mf W}$ satisfies $\mc N(\mf I_{n\times n} -\wt{\mf W})=\l\{c\mf 1,~c\in\mb R\r\}$, where $\mf 1$ is an all 1 vector in $\mb R^n$.
\end{enumerate}
These conditions are rather mild and satisfied by most doubly stochastic mixing matrices used in practice. Some specific examples are Markov transition matrices of max-degree chain and Metropolis-Hastings chain (see \cite{boyd2004fastest} for detailed discussions). 
Let $\mf x_i\in\mb R^d$ be the local variable on the node $i$. Define
\[
\mf x := \l[
\begin{tabular}{ l }
  $\mf x_1$ \\
  $\mf x_2$ \\
  $\vdots$ \\
  $\mf x_n$
\end{tabular}
\r]\in\mb R^{nd},~~
\mf b := \l[
\begin{tabular}{ l }
  $\mf b_1$ \\
  $\mf b_2$ \\
  $\vdots$ \\
  $\mf b_n$
\end{tabular}
\r]\in\mb R^{nd},~~
\mf A  = \l[
\begin{tabular}{lll}
  $\mf W_{11}$ & $\cdots$ & $\mf W_{1n}$\\
  $\vdots$ & $\ddots$   &  $\vdots$\\
  $\mf W_{n1}$ & $\cdots$ & $\mf W_{nn}$ 
\end{tabular}\r]\in\mb R^{(nd)\times(nd)},
\]
where 
$$\mf W_{ij} =
\begin{cases}
(1- \wt{w}_{ij})\mf I_{d\times d},~~&\text{if}~i=j\\
-\wt{w}_{ij}\mf I_{d\times d},~~&\text{if}~i\neq j
\end{cases},$$ 
and $\wt{w}_{ij}$ is $ij$-th entry of the mixing matrix $\widetilde{\mf W}$.
By the aforementioned null space property of the mixing matrix $\wt{\mf W}$, it is easy to see that the null space of the matrix $\mf A$ is
\begin{equation}\label{eq:null-space}
\mc N(\mf A) = \l\{ \mf u \in \mb R^{nd}:~\mf u = [\mf u_1^T,\cdots,\mf u_n^T]^T,~\mf u_1 = \mf u_2 = \cdots = \mf u_n \r\},
\end{equation}
Then, because of the null space property \eqref{eq:null-space}, one can equivalently write problem \eqref{eq:geo-median} in a ``distributed fashion'' as follows:
\begin{align}
\min&~~\sum_{i=1}^n\|\mf x_i - \mf b_i\|   \label{dec-prob-1}\\
s.t.&~~ \mf A\mf x = 0, \|\mf x_i - \mf b_i\|\leq D,~i=1,2,\cdots,n  \label{dec-prob-2},
\end{align}
where we set the constant $D$ to be large enough so that the solution belongs to the set $\mathcal{X}:=\l\{\mf x\in\mb R^{nd}: ~\|\mf x_i - \mf b_i\|\leq D,i=1,2,\cdots,n \r\}$. 
This is in the same form as (\ref{prob-1}-\ref{prob-2}) with $\mc X:= \{\mf x\in\mb R^{nd}:~\|\mf x_i - \mf b_i\|\leq D,~i=1,2,\cdots,n\}$.

\subsection{Distributed implementation}
In this section, we show how to implement the proposed algorithm to solve (\ref{dec-prob-1}-\ref{dec-prob-2}) in a distributed way. 
Let $\lambda_t = [\lambda_{t,1}^T,~\lambda_{t,2}^T,\cdots,~\lambda_{t,n}^T]\in\mb R^{nd}$, $\widehat{\lambda}_t = [\widehat{\lambda}_{t,1}^T,~\widehat{\lambda}_{t,2}^T,\cdots,~\widehat{\lambda}_{t,n}^T]\in\mb R^{nd}$ be the vectors of Lagrange multipliers  defined in Algorithm 1, where each $\lambda_{t,i},~\widehat{\lambda}_{t,i}\in\mb R^d$. Then, 
each agent $i\in\{1,2,\cdots,n\}$ in the network is responsible for updating the corresponding Lagrange multipliers $\lambda_{t,i}$ and $\widehat{\lambda}_{t,i}$ according to Algorithm 1, which has the initial values $\lambda_{0,i} = \lambda_{-1,i} = \widetilde{\lambda}_i$. Note that the first, third and fourth steps in Algorithm 1 are naturally separable regarding each agent. It remains to check if the second step can be implemented in a distributed way.

Note that in the second step, we obtain the primal update $\mf x(\widehat{\lambda}_t) = [\mf x_1(\widehat{\lambda}_t)^T,\cdots,\mf x_n(\widehat{\lambda}_t)^T]\in\mb R^{nd}$ by solving the following problem:
\[
\mf x(\widehat{\lambda}_t) = \text{argmax}_{\mf x:\|\mf x_i - \mf b_i\|\leq D,~i=1,2,\cdots,n}~-\dotp{\widehat{\lambda}_t}{\mathbf{A}\mathbf{x}} - \sum_{i=1}^n\l( \|\mathbf{x}_i - \mf b_i\| + \frac{\mu}{2}\|\mf x_i - \widetilde{\mf{x}}_i\|^2 \r),
\]
where $\widetilde{\mf{x}}_i\in\mb R^d$ is a fixed point in the feasible set. 
We separate the maximization according to different agent $i\in\{1,2,\cdots,n\}$:
\begin{align*}
\mf x_i(\widehat{\lambda}_t)
=&\text{argmax}_{\mf x_i:\|\mf x_i - \mf b_i\|\leq D}~-\sum_{j=1}^n\dotp{\wh\lambda_{t,j}}{\mf W_{ji}\mf x_i}- \|\mathbf{x}_i - \mf b_i\| - \frac{\mu}{2}\|\mf x_i - \widetilde{\mf{x}}_i\|^2.
\end{align*}
Note that according to the definition of $\mf W_{ji}$, it is equal to 0 if agent $j$ is not the neighbor of agent $i$. More specifically, Let $\mc N_i$ be the set of neighbors of agent $i$ (including the agent $i$ itself), then, the above maximization problem can be equivalently written as
\begin{align*}
&\text{argmax}_{\mf x_i:\|\mf x_i - \mf b_i\|\leq D} -\sum_{j\in\mc N_i} \dotp{\wh\lambda_{t,j}}{\mf W_{ji}\mf x_i}- \|\mathbf{x}_i - \mf b_i\| - \frac{\mu}{2}\|\mf x_i - \widetilde{\mf{x}}_i\|^2\\
=&\text{argmax}_{\mf x_i:\|\mf x_i - \mf b_i\|\leq D}  -\dotp{\sum_{j\in\mc N_i}\mf W_{ji}\wh\lambda_{t,j}}{\mf x_i}- \|\mathbf{x}_i - \mf b_i\| - \frac{\mu}{2}\|\mf x_i - \widetilde{\mf{x}}_i\|^2~~i\in\{1,2,\cdots,n\},
\end{align*}
where we used the fact that $\mf W_{ji}^T=\mf W_{ji}$.
Solving this problem only requires the local information from each agent. Completing the squares  gives 
\begin{equation}\label{eq:dec-sub}
\mf x_i(\widehat{\lambda}_t) = \text{argmax}_{\|\mf x_i - \mf b_i\|\leq D}-\frac{\mu}{2}\l\| \mf x_i - \l(\widetilde{\mf x}_i-\frac{1}{\mu}\sum_{j\in\mc N_i}\mf W_{ji}\wh\lambda_{t,j}  \r) \r\|^2 - \|\mf x_i - \mf b_i\|.
\end{equation}
The solution to such a subproblem has a closed form, as is shown in the following lemma (the proof is given in Supplement \ref{proof-explicit}):
\begin{lemma}\label{lem:sol-sub}
Let $\mf a_i = \widetilde{\mf x}_i-\frac{1}{\mu}\sum_{j\in\mc N_i}\mf W_{ji}\wh\lambda_{t,j}$, then, the solution to \eqref{eq:dec-sub} has the following closed form:
\[
\mf x_i(\widehat{\lambda}_t)=
\begin{cases}
\mf b_i,~~&\text{if}~~\|\mf b_i - \mf a_i\|\leq1/\mu,\\
\mf b_i - \frac{\mf b_i - \mf a_i}{\|\mf b_i - \mf a_i\|}\l(\|\mf b_i - \mf a_i\| - \frac{1}{\mu}\r),~~&\text{if}~~\frac{1}{\mu}< \|\mf b_i - \mf a_i\|\leq \frac{1}{\mu}+D,\\
 \mf b_i - \frac{\mf b_i - \mf a_i}{\|\mf b_i - \mf a_i\|}D,~~&\text{otherwise}.
\end{cases}
\]
\end{lemma}

\subsection{Local error bound condition}
The proof of the this theorem is given in Supplement \ref{proof-local-geo}. 
\begin{theorem}\label{thm:geo-local-bound}
The Lagrange dual function of (\ref{dec-prob-1}-\ref{dec-prob-2})  is non-smooth and given by the following
\[
F(\lambda) = - \dotp{\mf A^T\lambda}{\mf b} + D\sum_{i=1}^n(\|\mf{A}_{[i]}^T\lambda\|-1)\cdot I\l(\|\mf{A}_{[i]}^T\lambda\|>1\r),
\]
where 
$\mf A_{[i]}= [\mf W_{1i}~\mf W_{2i}~\cdots~\mf W_{ni}]^T
$
is the $i$-th column block of the matrix $\mf A$, $I\l(\|\mf{A}_{[i]}^T\lambda\|>1\r)$ is the indicator function which takes 1 if $\|\mf{A}_{[i]}^T\lambda\|>1$ and 0 otherwise. 
 Let $\Lambda^*$ be the set of optimal Lagrange multipliers defined according to \eqref{eq:optimal-set-1}. Suppose $D\geq 2n\cdot \max_{i,j\in\mc V}\|\mf b_i - \mf b_j\|$, then, for any $\delta>0$, there exists a $C_\delta>0$ such that 
\[
\text{dist}(\lambda,\Lambda^*)\leq C_\delta(F(\lambda)-F^*)^{1/2},~\forall \lambda\in \mc S_\delta.
\]
Furthermore, there exists a unique vector $\nu^*\in\mb R^{nd}$ s.t. $\mc P_{\mf A}\lambda^* = \nu^*,~\forall \lambda^*\in\Lambda^*$, i.e. Assumption \ref{assumption}(d) holds. Thus, applying the proposed method gives the convergence time $\mc O\l(\varepsilon^{-4/5}\log_2(\varepsilon^{-1})\r)$.
\end{theorem}

\section{Simulation Experiments}\label{sec:simulation}
In this section, we conduct simulation experiments on the distributed geometric median problem. Each vector $\mf b_i\in\mathbb R^{100},~i\in\{1,2,\cdots,n\}$ is sampled from the uniform distribution in $[0,10]^{100}$, i.e. each entry of $\mf b_i$ is independently sampled from uniform distribution on $[0,10]$. We compare our algorithm with DSM (\cite{nedic2009distributed}), 
P-EXTRA (\cite{shi2015proximal}), Jacobian parallel ADMM (\cite{deng2017parallel}) and Smoothing (\cite{necoara2008application}) under different network sizes ($n=20,50,100$). Each network is randomly generated with a particular connectivity ratio\footnote{The connectivity ratio is defined as the number of edges divided by the total number of possible edges $n(n+1)/2$.}, and the mixing matrix is chosen to be the Metropolis-Hastings Chain (\cite{boyd2004fastest}), which can be computed in a distributed manner.
We use the relative error as the performance metric, which is defined as $\|\overline{\mf x}_t -\mf x^* \|/\|\mf x_0 - \mf x^*\|$ for each iteration $t$. The vector $\mf x_0\in\mb R^{nd}$ is the initial primal variable. The vector $\mf x^*\in\mb R^{nd}$ is the optimal solution computed by CVX \cite{grant2008cvx}. For our proposed algorithm, $\overline{\mf x}_t$ is the restarted primal average up to the current iteration. For all other algorithms, $\overline{\mf x}_t$ is the primal average up to the current iteration. The results are shown below. We see in all cases, our proposed algorithm is much better than, if not comparable to, other algorithms. For detailed simulation setups and additional simulation results, see Supplement \ref{add-simulation}. 
\begin{figure*}[ht!] 
    \centering
    \begin{subfigure}[t]{0.28\textwidth}
        \centering
        \includegraphics[height=3cm] {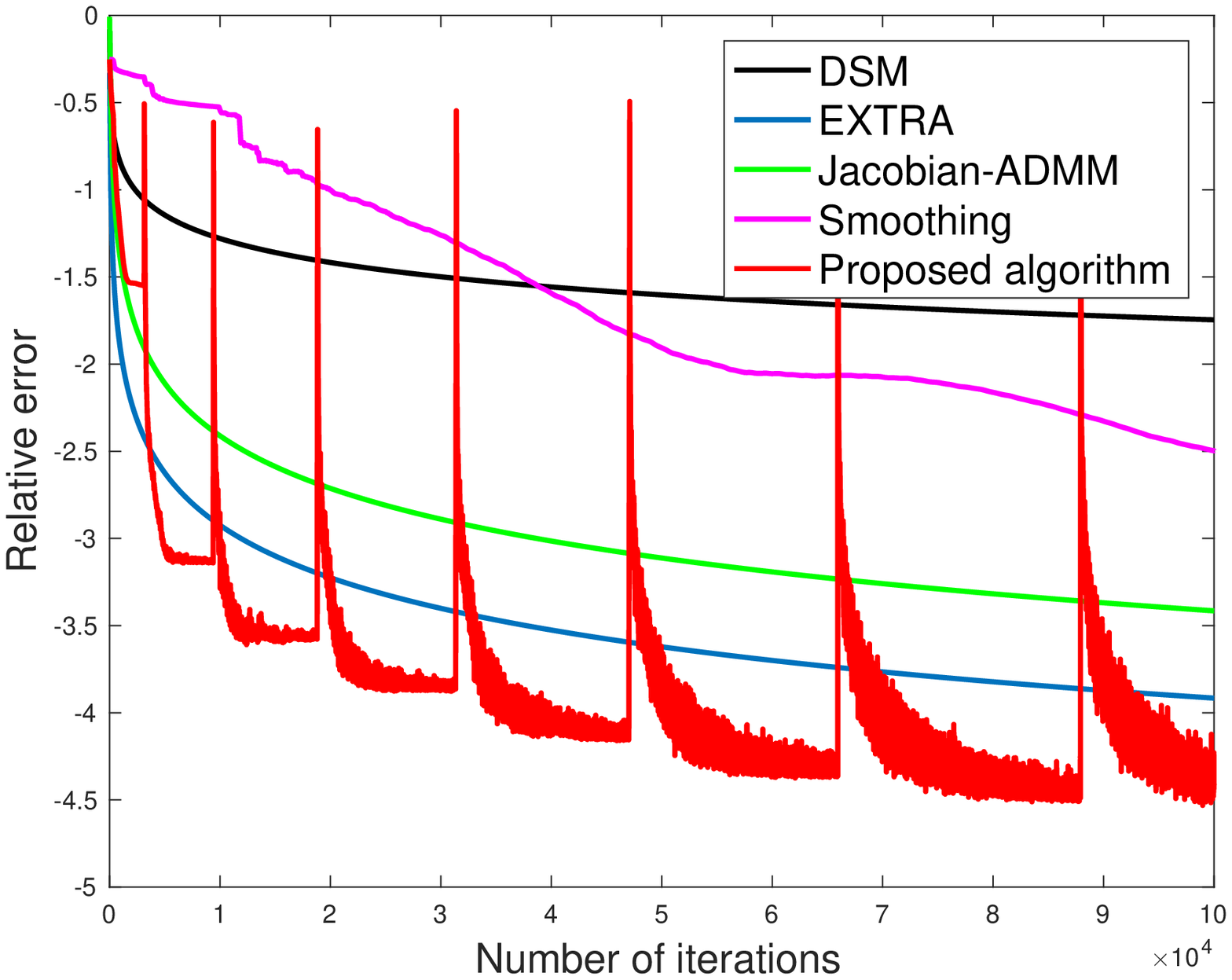}
        \vspace{-1.5em}
        \caption{}
    \end{subfigure}%
    ~ 
    \begin{subfigure}[t]{0.28\textwidth}
        \centering
        \includegraphics[height=3cm] {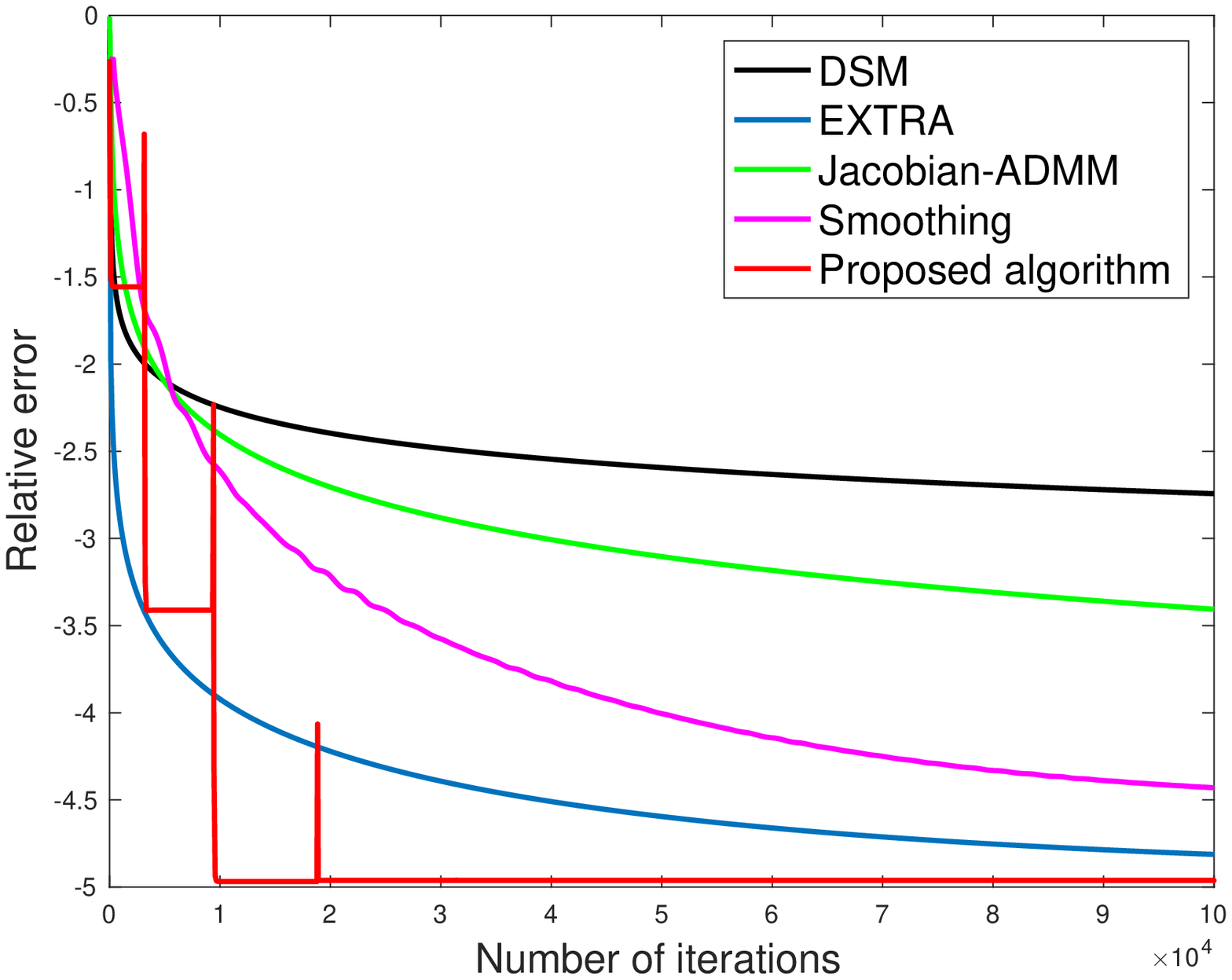}
        \vspace{-1.5em}
        \caption{}
    \end{subfigure}
    ~
    \begin{subfigure}[t]{0.28\textwidth}
        \centering
        \includegraphics[height=3cm] {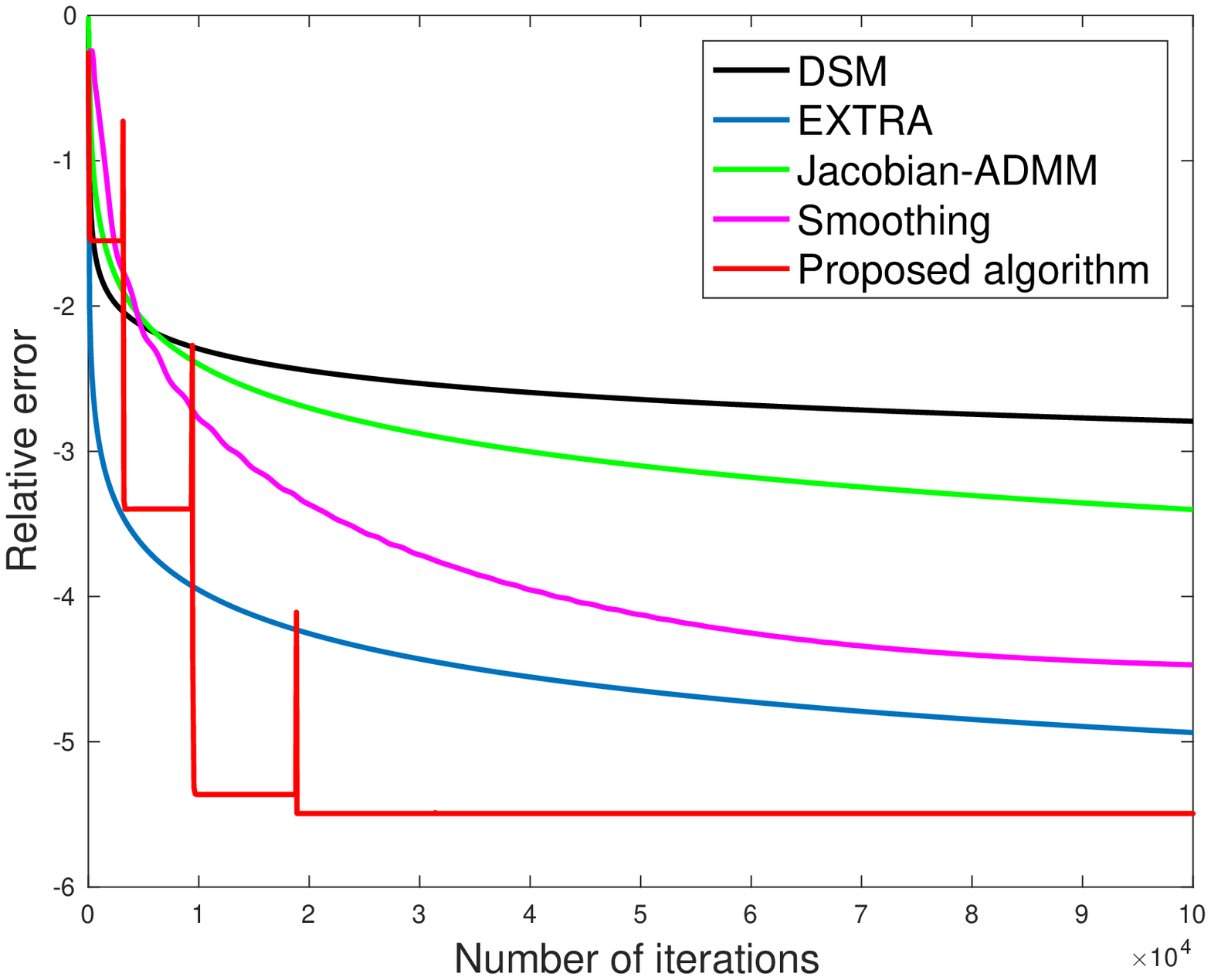}
        \vspace{-1.5em}
        \caption{}
    \end{subfigure}
 
    \caption{Comparison of different algorithms on networks of different sizes. (a) $n=20$, connectivity ratio=0.15. (b) $n=50$, connectivity ratio=0.13. (c) $n=100$, connectivity ratio=0.1.}
\end{figure*}

\subsubsection*{Acknowledgments}
The authors thank Stanislav Minsker and Jason D. Lee for helpful discussions related to the geometric median problem.
Qing Ling's research is supported in part by the National Science Foundation China under Grant 61573331 and Guangdong IIET Grant 2017ZT07X355.
Michael J. Neely's research is supported in part by the National Science Foundation under Grant CCF-1718477.

\bibliographystyle{chicago}
\bibliography{bibliography}

\newpage

\normalsize
\section{Supplement}

\subsection{Smoothing lemma}
In this section, we show that adding the strongly convex term on the primal indeed gives a smoothed dual.
\begin{lemma}\label{lem:smooth}
Let $f_\mu(\mf x)$ be defined as above and let $g_i:\mc X\rightarrow \mb R,~i=1,2,\cdots,N$ be a sequence of $G$-Lipschitz continuous convex functions, i.e. $\|\mf g(\mf x) - \mf g(\mf y)\|\leq G\|\mf x - \mf y\|,~\forall \mf x, \mf y\in\mc X$, where $\mf g(\mf x) = [g_1(\mf x),\dots,g_N(\mf x)]$. Then, the Lagrange dual function
\[
 d_\mu(\lambda):=\max_{\mf x\in\mc X} - \dotp{\lambda}{\mf g(\mf x)} - f_\mu(\mf x),~\lambda\in\mb R^N
\]
is smooth with modulus $G^2/\mu$. In particular, if $\mf g(\mf x) = \mf A\mf x - \mf b$, then, the smooth modulus is equal to
$\sigma_{\max}(\mf A^T\mf A)/\mu$, where $\sigma_{\max}(\mf A^T\mf A)$ denotes the maximum eigenvalue of $\mf A^T\mf A$.
\end{lemma}
This proof of this lemma is rather standard (see also proof of Lemma 6 of \cite{yu2017convergence}) and the special case of $\mf g(\mf x) = \mf A\mf x - \mf b$ can also be derived from Fenchel duality (\cite{beck20141}). 

\begin{proof}[Proof of Lemma \ref{lem:smooth}]
First of all, note that the function $h_\lambda(\mf x) = - \dotp{\lambda}{\mf g(\mf x)} - f_\mu(\mf x)$ is strongly concave, it follows that there exists a unique minimizer $\mf x(\lambda):= \text{argmax}_{\mf x\in\mc X}h_\lambda(\mf x)$. By Danskin's theorem (see \cite{bertsekas1999nonlinear} for details), we have for any 
$\lambda\in\mb R^N$,
\[
\nabla d_\mu(\lambda) = \mf g (\mf x(\lambda)).
\]
Now, consider any $\lambda_1,\lambda_2\in\mb R^N$, we have
\begin{equation}\label{inter-1}
\|\nabla d_\mu(\lambda_1) - \nabla d_\mu(\lambda_2)\| = \|\mf g (\mf x(\lambda_1)) - \mf g (\mf x(\lambda_2))\|
\leq G\|x(\lambda_1) - x(\lambda_2)\|.
\end{equation}
where the equality follows from Danskin's Theorem and the inequality follows from Lipschitz continuity of $\mf g(\mf x)$.
Again, by the fact that $h_\mu(\mf x)$ is strongly concave with modulus $\mu$,
\begin{align*}
h_{\lambda_1}(\mf x(\lambda_2))\leq  h_{\lambda_1}(\mf x(\lambda_1)) - \frac\mu2\|x(\lambda_1) - x(\lambda_2)\|^2,\\
h_{\lambda_2}(\mf x(\lambda_1))\leq  h_{\lambda_2}(\mf x(\lambda_2)) - \frac\mu2\|x(\lambda_1) - x(\lambda_2)\|^2,
\end{align*} 
which implies
\begin{align*}
-\dotp{\lambda_1}{\mf g(\mf x(\lambda_2))} - f_\mu(\mf x(\lambda_2))\leq  -\dotp{\lambda_1}{\mf g(\mf x(\lambda_1))}- f_\mu(\mf x(\lambda_1)) - \frac\mu2\|x(\lambda_1) - x(\lambda_2)\|^2,\\
-\dotp{\lambda_2}{\mf g(\mf x(\lambda_1))} - f_\mu(\mf x(\lambda_1))\leq -\dotp{\lambda_2}{\mf g(\mf x(\lambda_2)} - f_\mu(\mf x(\lambda_2)) - \frac\mu2\|x(\lambda_1) - x(\lambda_2)\|^2.
\end{align*} 
Adding the two inequalities gives
\begin{align*}
\mu\|x(\lambda_1)) - x(\lambda_2))\|^2\leq& \dotp{\lambda_1-\lambda_2}{\mf g(\mf x(\lambda_1)) - \mf g(\mf x(\lambda_2))}\\
\leq&\|\lambda_1-\lambda_2\|\cdot\|\mf g(\mf x(\lambda_1)) - \mf g(\mf x(\lambda_2))\|\\
\leq&G\|\lambda_1-\lambda_2\|\cdot\|x(\lambda_1)) - x(\lambda_2))\|,
\end{align*}
where the last inequality follows from Lipschitz continuity of $\mf g(\mf x)$ again.
This implies
\[
\|x(\lambda_1) - x(\lambda_2)\|\leq \frac{G}{\mu}\|\lambda_1-\lambda_2\|.
\]
Combining this inequality with \eqref{inter-1} gives
\[
\|\nabla d_\mu(\lambda_1) - \nabla d_\mu(\lambda_2)\| \leq \frac{G^2}{\mu}\|\lambda_1 - \lambda_2\|,
\]
finishing the first part of the proof. The second part of the claim follows easily from the fact that 
$\|\mf A \mf x - \mf A\mf y\|\leq \sqrt{\sigma_{\max}(\mf A^T\mf A)}\|\mf x - \mf y\|$.
\end{proof}

\subsection{Proof of Theorem \ref{main-thm-1}}\label{proof-each-stage}
In this section, we give a convergence time proof of each stage. As a preliminary,  
we have the following basic lemma which bounds the perturbation of the Lagrange dual due to the primal smoothing. 
\begin{lemma}\label{lem:perturbation}
Let $F(\lambda)$ and $F_{\mu}(\lambda)$ be functions defined in \eqref{eq:dual-function} and \eqref{eq:dual-function-2}, respectively. 
Then, we have for any $\lambda\in\mb R^N$,
\[
0\leq F(\lambda) - F_\mu(\lambda) \leq \mu D^2/2
\]
and
\[
0\leq F(\lambda^*)-F_\mu(\lambda_\mu^*)\leq \mu D^2/2,
\]
for any $\lambda^*\in\Lambda^*$ and $\lambda_\mu^*\in\Lambda_\mu^*$.
\end{lemma}
\begin{proof}[Proof of Lemma \ref{lem:perturbation}]
First of all, for any $\lambda\in\mb R^N$, define 
\begin{align*}
h(\mf x)&:= - \dotp{\lambda}{\mf A \mf x - \mf b} - f(\mf x),\\
h_\mu(\mf x)& := - \dotp{\lambda}{\mf A \mf x - \mf b} - f_\mu(\mf x).
\end{align*}
Then, let
\begin{align*}
\mf x(\lambda) &\in \text{argmax}_{\mf x\in\mc X} ~h(\mf x),\\
\mf x_\mu(\lambda) &\in \text{argmax}_{\mf x\in\mc X}~h_{\mu}(\mf x) ,
\end{align*}
and we have for any $\lambda\in\mb R^N$,
\begin{align*}
F(\lambda) - F_\mu(\lambda) =& h(\mf x(\lambda)) - h_\mu(\mf x_\mu(\lambda))\\
=& h(\mf x(\lambda)) - h_\mu(\mf x(\lambda)) + h_\mu(\mf x(\lambda)) - h_\mu(\mf x_\mu(\lambda))\\
\leq& h(\mf x(\lambda)) - h_\mu(\mf x(\lambda)) \\
=& f_\mu(\mf x(\lambda)) - f(\mf x(\lambda))\leq \mu D^2/2,
\end{align*}
where the first inequality follows from the fact that $\mf x_\mu(\lambda)$ maximizes $h_\mu(\lambda)$. Similarly, we have
\begin{align*}
F_\mu(\lambda) - F(\lambda) =& h_\mu(\mf x_\mu(\lambda)) - h(\mf x(\lambda))\\
=& h_\mu(\mf x_\mu(\lambda)) - h(\mf x_\mu(\lambda)) + h(\mf x_\mu(\lambda)) - h(\mf x(\lambda))\\
\leq& h_\mu(\mf x_\mu(\lambda)) - h(\mf x_\mu(\lambda))\\
=& f(\mf x(\lambda)) - f_\mu(\mf x(\lambda)) \leq 0,
\end{align*}
where the first inequality follows from the fact that $\mf x(\lambda)$ maximizes $h(\lambda)$. 
Furthermore, we have
\begin{align*}
&F(\lambda^*)-F_\mu(\lambda_\mu^*) = F(\lambda^*)-F(\lambda_\mu^*) + F(\lambda_\mu^*) -  F_\mu(\lambda_\mu^*)
\leq F(\lambda_\mu^*) -  F_\mu(\lambda_\mu^*)\leq\mu D^2/2, \\
&F_\mu(\lambda_\mu^*)-F(\lambda^*) = F_\mu(\lambda_\mu^*)- F_\mu(\lambda^*) + F_\mu(\lambda^*)- F(\lambda^*)
\leq F_\mu(\lambda^*)- F(\lambda^*)\leq 0,
\end{align*}
finishing the proof.
\end{proof}

To prove Theorem \ref{main-thm-1},
we start by rewriting the primal-dual smoothing algorithm (Algorithm 1) as the Nesterov's accelerated gradient algorithm on the smoothed dual function $F_\mu(\lambda)$: For any $t=0,1,\cdots,T-1$,
\begin{equation}\label{eq:acc-gradient}
\left \{
  \begin{tabular}{c}
   $\widehat{\lambda}_t = \lambda_t + \theta_t(\theta_{t-1}^{-1} - 1)(\lambda_t - \lambda_{t-1})$ \\
   $\lambda_{t+1} = \widehat{\lambda}_t - \mu\nabla F_\mu(\widehat{\lambda}_t)$~~~~~~~~~~~~~~~~~~~~\\
   $\theta_{t+1} = \frac{\sqrt{\theta_t^4 + 4\theta_t^2} - \theta_t^2}{2}$~~~~~~~~~~~~~~~~~~~~~~~~
  \end{tabular}
\right.
\end{equation}
where we use Danskin's Theorem to claim that $\nabla F_\mu(\widehat{\lambda}_t) = \mf b - \mf A \mf x(\widehat{\lambda}_t)$. As $t\rightarrow\infty$, we have $\frac{\theta_t}{\theta_{t-1}} = \sqrt{1-\theta_t}\rightarrow1$.
Classical results on the convergence time of accelerated gradient methods are as follows:

\begin{theorem}[Theorem 1 of \cite{tseng2010approximation}]\label{thm-Tseng}
Consider the algorithm \eqref{eq:acc-gradient} starting from $\lambda_0 = \lambda_{-1} = \wt{\lambda}$. 
For any $\lambda\in\mb R^N$, we have
\begin{equation}\label{eq:dual-bound}
F_\mu(\lambda_t)\leq F_\mu(\lambda) + \theta_{t-1}^2\frac{\sigma_{\max}(\mf A^T\mf A)\|\lambda - \wt\lambda\|^2}{\mu},
\end{equation}
Furthermore, for any slot $t\in\{0,1,2,\cdots,T-1\}$,
\begin{multline}\label{eq:one-step}
F_\mu(\lambda_{t+1}) \leq (1-\theta_t)\l( F_\mu(\wh\lambda_t) + \dotp{\nabla F_\mu(\wh\lambda_t)}{\lambda_t- \wh\lambda_t} \r)
+\theta_t\l( F_\mu(\wh\lambda_t) + \dotp{\nabla F_\mu(\wh\lambda_t)}{\lambda- \wh\lambda_t} \r)\\
+\frac{\theta_t^2\sigma_{\max}(\mf A^T\mf A)}{2\mu} \l(\| \lambda - \mf z_t \|^2 - \| \lambda - \mf z_{t+1} \|^2\r),
\end{multline}
where $\mf z_t = -(\theta_t^{-1} - 1)\lambda_t + \theta_t^{-1}\wh\lambda_t$.
\end{theorem}

This theorem bounds the convergence time of the dual function. Our goal is to pass this dual convergence result to that of primal objective and constraint.
Specifically, we aim to show the following primal objective bound and constraint violation:

To prove Theorem \ref{main-thm-1}, we start by proving the following bound:
\begin{lemma}\label{lem:supp-1}
Consider running Algorithm 1 with a given initial condition $\widetilde{\lambda}$ in $\mathbb{R}^N$.
For any $\lambda\in\mb R^N$, we have
\begin{equation}\label{eq:one-step-primal}
f_\mu(\overline{\mf x}_T) - \dotp{\mf b - \mf A\overline{\mf x}_T}{\lambda} - f^*_\mu 
\leq \frac{\sigma_{\max}(\mf A^T\mf A)}{2\mu S_T}\l( \|\lambda-\wt\lambda\|^2 - \| \lambda  - \mf z_T \|^2 \r),
\end{equation}
where $\mf z_T$ is defined in Theorem  \ref{thm-Tseng},
\[
f_\mu^* := \min_{\mf A\mf x - \mf b = 0,~\mf x\in\mc X}f_\mu(x),~~\overline{\mf x}_T:= \frac{1}{S_T}\sum_{t=0}^{T-1}\frac{\mf x(\widehat{\lambda}_t)}{\theta_t},
\]
\end{lemma}
\begin{proof}[Proof of Lemma \ref{lem:supp-1}]
First,
subtracting $F_\mu(\lambda_\mu^*)$ from both sides of \eqref{eq:one-step} in Theorem \ref{thm-Tseng}, we have for any $\lambda\in\mb R^N$ and any $t\in\{0,1,2,\cdots,T-1\}$,
\begin{align*}
F_\mu(\lambda_{t+1}) - F_\mu(\lambda_\mu^*)
\leq& (1-\theta_t)\l(F_\mu(\wh\lambda_t)+ \dotp{\nabla F_\mu(\wh\lambda_t)}{\lambda_t- \wh\lambda_t} - F_\mu(\lambda_\mu^*)\r)\\
&+ \theta_t \l(F_\mu(\wh\lambda_t)+\dotp{\nabla F_\mu(\wh\lambda_t)}{ \lambda - \wh\lambda_t} - F_\mu( \lambda_\mu^*)\r)\\
&+\frac{\theta_t^2\sigma_{\max}(\mf A^T\mf A)}{2\mu}\l(\|  \lambda - \mf z_t \|^2 - \|  \lambda - \mf z_{t+1} \|^2\r)\\
\leq&(1-\theta_t)\l(F_\mu(\lambda_t) - F_\mu(\lambda_\mu^*)\r) + \theta_t  \l(F_\mu(\wh\lambda_t)+\dotp{\nabla F_\mu(\wh\lambda_t)}{\lambda - \wh\lambda_t} - F_\mu(\lambda_\mu^*)\r)\\
&+\frac{\theta_t^2\sigma_{\max}(\mf A^T\mf A)}{2\mu}\l(\| \lambda - \mf z_t \|^2 - \| \lambda - \mf z_{t+1} \|^2\r),
\end{align*}
where the second inequality follows from the convexity of $F_\mu$ that $F_\mu(\wh\lambda_t)+ \dotp{\nabla F_\mu(\wh\lambda_t)}{\lambda_t- \wh\lambda_t}\leq F_\mu(\lambda_t)$.
Dividing $\theta_t^2$ from both sides gives $\forall t\geq1$,
\begin{align}
\frac{1}{\theta_t^2}\l(F_\mu(\lambda_{t+1}) - F_\mu(\lambda_\mu^*)\r)
\leq &\frac{1-\theta_t}{\theta_t^2}\l(F_\mu(\lambda_t) - F_\mu(\lambda_\mu^*)\r) + \frac{1}{\theta_t}  \l(F_\mu(\wh\lambda_t)+\dotp{\nabla F_\mu(\wh\lambda_t)}{ \lambda - \wh\lambda_t} - F_\mu(\lambda_\mu^*)\r) \nonumber\\
&+\frac{\sigma_{\max}(\mf A^T\mf A)}{2\mu}\l(\|  \lambda - \mf z_t \|^2 - \|  \lambda - \mf z_{t+1} \|^2\r)\nonumber\\
=&\frac{1}{\theta_{t-1}^2}\l(F_\mu(\lambda_t) - F_\mu(\lambda_\mu^*)\r) + \frac{1}{\theta_t}  \l(F_\mu(\wh\lambda_t)+\dotp{\nabla F_\mu(\wh\lambda_t)}{ \lambda - \wh\lambda_t} - F_\mu(\lambda_\mu^*)\r) \nonumber\\
&+\frac{\sigma_{\max}(\mf A^T\mf A)}{2\mu}\l(\|  \lambda - \mf z_t \|^2 - \|  \lambda - \mf z_{t+1} \|^2\r), \label{eq:star}
\end{align}
where the last equality uses the identity $(1-\theta_t)/\theta_t^2 = 1/\theta_{t-1}^2$.    On the other hand, applying equation \eqref{eq:star} at $t=0$ and using $\theta_0=\theta_{-1}=1$ gives 
$(1-\theta_0)/\theta_0^2 = 0$ and 
\begin{align*}
\frac{1}{\theta_0^2}\l(F_\mu(\lambda_{1}) - F_\mu(\lambda_\mu^*)\r)
\leq & \frac{1}{\theta_0}  \l(F_\mu(\wh\lambda_0)+\dotp{\nabla F_\mu(\wh\lambda_t)}{ \lambda - \wh\lambda_0} - F_\mu(\lambda_\mu^*)\r)\\
&+\frac{\sigma_{\max}(\mf A^T\mf A)}{2\mu}\l(\|  \lambda - \mf z_0 \|^2 - \|  \lambda - \mf z_{1} \|^2\r).
\end{align*}
Taking telescoping sums from both sides from $t=0$ to $t = T-1$ gives
\begin{multline*}
0\leq \frac{1}{\theta_{T-1}^2}\l(F_\mu(\lambda_{T}) - F_\mu(\lambda_\mu^*)\r)
\leq\sum_{t=0}^{T-1}\frac{1}{\theta_t} \l(F_\mu(\wh\lambda_t)+\dotp{\nabla F_\mu(\wh\lambda_t)}{ \lambda - \wh\lambda_t} - F_\mu(\lambda_\mu^*)\r)\\
+\frac{\sigma_{\max}(\mf A^T\mf A)}{2\mu}\l(\|  \lambda - \mf z_0 \|^2 - \|  \lambda - \mf z_{T} \|^2\r).
\end{multline*}
By Assumption \ref{assumption}(a), the feasible set $\l\{  \mathbf {Ax} -\mathbf{b} =0\r\}$ is not empty, and thus, strong duality holds for problem 
\[\min_{\mf A\mf x - \mf b = 0,~\mf x\in\mc X}f_\mu(x)\]
(See, for example Proposition 5.3.1 of \cite{bertsekas2009convex}),
and we have $F_\mu(\lambda_\mu^*) = -f_\mu^*$. Since 
$$\nabla F_\mu(\wh\lambda_t) = \mf b - \mf A \mf x(\wh\lambda_t),~~
F_\mu(\wh\lambda_t)=\dotp{\wh\lambda_t}{\mf b- \mf A\mf x(\wh\lambda_t)} - f_\mu(\mf x(\wh\lambda_t)),$$ 
it follows,
\begin{align*}
0\leq&\sum_{t=0}^{T-1}\frac{1}{\theta_t} \l(\dotp{\wh\lambda_t}{\mf b- \mf A\mf x(\wh\lambda_t)} - f_\mu(\mf x(\wh\lambda_t))
+\dotp{\mf b - \mf A \mf x(\wh\lambda_t)}{ \lambda - \wh\lambda_t} + f_\mu^*\r)\\
&+\frac{\sigma_{\max}(\mf A^T\mf A)}{2\mu}\l(\|  \lambda - \mf z_0 \|^2 - \|  \lambda - \mf z_{T} \|^2\r)\\
=&\sum_{t=0}^{T-1}\frac{1}{\theta_t} \l( - f_\mu(\mf x(\wh\lambda_t))
+\dotp{\mf b - \mf A \mf x(\wh\lambda_t)}{ \lambda} + f_\mu^*\r)
+\frac{\sigma_{\max}(\mf A^T\mf A)}{2\mu}\l(\|  \lambda - \mf z_0 \|^2 - \|  \lambda - \mf z_{T} \|^2\r)
\end{align*}
Rearranging the terms and divding $S_T = \sum_{t=0}^{T-1}\frac{1}{\theta_t}$ from both sides,
\begin{multline*}
\frac{1}{S_T}\sum_{t=0}^{T-1}\frac{1}{\theta_t}\l(f_\mu(\mf x(\wh\lambda_t))
-\dotp{\mf b - \mf A \mf x(\wh\lambda_t)}{ \lambda} - f_\mu^*\r)\leq \frac{\sigma_{\max}(\mf A^T\mf A)}{2\mu S_T}\l(\|  \lambda - \mf z_0 \|^2 - \|  \lambda - \mf z_{T} \|^2\r).
\end{multline*}
Note that $\mf z_0 = \wt\lambda$ by the definition of $\mf z_t$. 
By Jensen's inequality, we can move the weighted average inside the function $f_\mu$ and finish the proof.
\end{proof}

\begin{proof}[Proof of Theorem \ref{main-thm-1}]
First of all, we have  by definition of $\Lambda_\mu^*$ in \eqref{eq:optimal-set-2} and strong duality, 
for any $\lambda_\mu^*\in\Lambda_\mu^*$,
\[
f_\mu(\overline{\mf x}_T) + \dotp{\mf A \overline{\mf x}_T - \mf b}{\lambda_\mu^*}\geq f_\mu^*.
\]
Substituting this bound into \eqref{eq:one-step-primal} gives 
\[
\dotp{\mf A \overline{\mf x}_T - \mf b}{\lambda-\lambda_\mu^*}\leq \frac{\sigma_{\max}(\mf A^T\mf A)}{2\mu S_T}\l(\|  \lambda - \wt\lambda \|^2 - \|  \lambda - \mf z_{T} \|^2\r)\leq  \frac{\sigma_{\max}(\mf A^T\mf A)}{2\mu S_T}\|  \lambda - \wt\lambda \|^2.
\]
Since this holds for any $\lambda\in\mb R^N$, the following holds:
\[
\max_{\lambda\in\mb R^N}\left[\dotp{\mf A \overline{\mf x}_T - \mf b}{\lambda-\lambda_\mu^*} -  \frac{\sigma_{\max}(\mf A^T\mf A)}{2\mu S_T}\|  \lambda - \wt\lambda \|^2\right]\leq0.
\]
The maximum is attained at $\lambda = \wt\lambda +\frac{\mu S_T}{\sigma_{\max}(\mf A^T\mf A)}\l( \mf A \overline{\mf x}_T - \mf b \r)$, which implies,
\begin{align*}
&\dotp{\mf A \overline{\mf x}_T - \mf b}{\wt\lambda-\lambda_\mu^*} + \frac{\mu S_T}{2\sigma_{\max}(\mf A^T\mf A)}\| \mf A \overline{\mf x}_T - \mf b \|^2\leq0.\\
\Rightarrow&\dotp{\mf A \overline{\mf x}_T - \mf b}{\mc P_{\mf A}\l(\wt\lambda-\lambda_\mu^*\r)} + \frac{\mu S_T}{2\sigma_{\max}(\mf A^T\mf A)}\| \mf A \overline{\mf x}_T - \mf b \|^2\leq0,
\end{align*}
where we used the fact that $\mf A \overline{\mf x}_T - \mf b = \mc P_{\mf A}(\mf A \overline{\mf x}_T - \mf b)$ because $\mathbf{b}$ is in the column space of $\mathbf{A}$. 
By Cauchy-Schwarz inequality, we have
\begin{align*}
\frac{\mu S_T}{2\sigma_{\max}(\mf A^T\mf A)}&\| \mf A \overline{\mf x}_T - \mf b \|^2\leq 
\|\mf A \overline{\mf x}_T - \mf b\| \cdot \|\mc P_{\mf A}\l(\wt\lambda-\lambda_\mu^*\r)\| \\
\Rightarrow&\| \mf A \overline{\mf x}_T - \mf b \|\leq \frac{2\sigma_{\max}(\mf A^T\mf A)}{\mu S_T}  \|\mc P_{\mf A}\l(\wt\lambda-\lambda_\mu^*\r)\|.
\end{align*}
Let $\wt\lambda^* = \text{argmin}_{\lambda^*\in\Lambda^*}\|\lambda^* - \wt\lambda\|$, by triangle inequality,
\begin{align*}
\|\mf A \overline{\mf x}_T - \mf b\|
\leq& 
\frac{2\sigma_{\max}(\mf A^T\mf A)}{\mu S_T}\l(  \|\mc P_{\mf A}\l(\wt\lambda-\wt\lambda^*\r)\| + \|\mc P_{\mf A}\l(\wt\lambda^* -\lambda_\mu^* \r)\| \r)\\
\leq&\frac{2\sigma_{\max}(\mf A^T\mf A)}{\mu S_T}\l(  \|\wt\lambda-\wt\lambda^*\| + \|\mc P_{\mf A}\l(\wt\lambda^* -\lambda_\mu^* \r)\| \r),
\end{align*}
where the second inequality follows from the non-expansiveness of the projection.
Now we look at the second term on the right hand side of the above inequality, Using Assumption \ref{assumption}(d), there exists a unique vector 
$\nu^*$ such that $\mc P_{\mf A}\lambda^*=\nu^*,~\forall \lambda^*\in\Lambda^*$. Thus,
\begin{multline*}
\|\mc P_{\mf A}\l(\wt\lambda^* -\lambda_\mu^* \r)\| = \|\nu^*-\mc P_{\mf A}\lambda_\mu^* \|= \min_{\lambda^*\in\Lambda:\mc P_{\mf A}\lambda^*=\nu^*}
\|\mc P_{\mf A}\l(\lambda^* -\lambda_\mu^* \r)\|\\
\leq \min_{\lambda^*\in\mb R^N:\mc P_{\mf A}\lambda^*=\nu^*}
\|\lambda^* -\lambda_\mu^* \| = \text{dist}(\lambda_\mu^*,\Lambda^*).
\end{multline*}
Thus, we get the constraint violation bound
\[
\|\mf A \overline{\mf x}_T - \mf b\|\leq  \frac{2\sigma_{\max}(\mf A^T\mf A)}{\mu S_T}\l(  \|\wt\lambda-\wt\lambda^*\| + \text{dist}(\lambda_\mu^*,\Lambda^*) \r).
\]
To get the objective suboptimality bound, we start from \eqref{eq:one-step-primal} again. Substituting $\lambda = \wt\lambda^* =  \text{argmin}_{\lambda^*\in\Lambda^*}\|\lambda^* - \wt\lambda\|$ into \eqref{eq:one-step-primal} gives
\[
f_\mu(\overline{\mf x}_T) - \dotp{\mf b - \mf A\overline{\mf x}_T}{\wt\lambda^*} - f^*_\mu 
\leq \frac{\sigma_{\max}(\mf A^T\mf A)}{2\mu S_T}\l( \|\wt\lambda^* -\wt\lambda\|^2 - \| \wt\lambda^*  - \mf z_T \|^2 \r)
\leq \frac{\sigma_{\max}(\mf A^T\mf A)}{2\mu S_T} \|\wt\lambda^* -\wt\lambda\|^2.
\]
By Cauchy-Schwarz inequality and the fact that $ \mf A\overline{\mf x}_T-\mathbf{b} = \mathcal{P}_{\mathbf{A}}( \mf A\overline{\mf x}_T-\mathbf{b})$, we have
\[
f_\mu(\overline{\mf x}_T) - f^*_\mu \leq \|\mf b - \mf A\overline{\mf x}_T\|\|\mathcal{P}_{\mf A}\wt\lambda^*\|+\frac{\sigma_{\max}(\mf A^T\mf A)}{2\mu S_T} \|\wt\lambda^* -\wt\lambda\|^2.
\]
By the fact that $f(\overline{\mf x}_T)\leq f_\mu( \overline{\mf x}_T) \leq f(\overline{\mf x}_T)+\frac{\mu}{2}D^2$, and the fact that 
$-f_\mu^* = F_\mu(\lambda_\mu^*)\geq F(\lambda^*)-\frac{\mu}{2}D^2 = -f^* -\frac{\mu}{2}D^2$ (from Lemma \ref{lem:perturbation}), we obtain
\[
f(\overline{\mf x}_T) - f^*\leq \|\mf b - \mf A\overline{\mf x}_T\|\|\mathcal{P}_{\mf A}\wt\lambda^*\|+\frac{\sigma_{\max}(\mf A^T\mf A)}{2\mu S_T} \|\wt\lambda^* -\wt\lambda\|^2+  \frac{\mu}{2}D^2,
\]
finishing the proof.
\end{proof}

\subsection{Proof of Theorem \ref{main-thm-2}}\label{proof-overall}
In this section, we give an analysis of the proposed homotopy method building upon the previous results on the primal-dual smoothing. Our improved convergence time analysis under such a homotopy method is built upon previous results, notably the following lemma:

\begin{lemma}[\cite{yang2015rsg}]\label{lem:yang-lin}
Consider any convex function $F:\mathbb{R}^N\rightarrow\mb R$ such that the set of optimal points $\Lambda^*$ defined in \eqref{eq:optimal-set-1} is non-empty. Then, for any  $\lambda\in\mb R^N$ and any $\varepsilon>0$,
\[ 
\|\lambda- \lambda_\varepsilon^\dagger\|\leq \frac{\text{dist}(\lambda_\varepsilon^\dagger,\Lambda^*)}{\varepsilon}
\l( F(\lambda) - F(\lambda_\varepsilon^\dagger) \r),
\]
where $\lambda_\varepsilon^\dagger:= \text{argmin}_{\lambda_\varepsilon\in\mc S_\varepsilon}\|\lambda - \lambda_\varepsilon\|$, and 
$\mc S_\varepsilon$ is the $\varepsilon$-sublevel set defined in \eqref{eq:sublevel}.
\end{lemma}

We start with the following easy corollary of Theorem \ref{thm-Tseng}.
\begin{corollary}
Suppose $\{\lambda_t\}_{t=0}^{T}$ is the sequence produced by Algorithm 1 with the initial condition $\lambda_0 = \lambda_{-1} = \wt{\lambda}$, then,
for any $\lambda\in\mb R^N$, we have
\begin{equation}\label{eq:dual-bound-2}
F(\lambda_t)\leq F(\lambda) + \theta_{t-1}^2\frac{\sigma_{\max}(\mf A^T\mf A)\|\lambda - \wt\lambda\|^2}{\mu} + \frac{D^2}{2}\mu,
\end{equation}
\end{corollary}
The proof of this corollary is obvious combining \eqref{eq:dual-bound} of Theorem \ref{thm-Tseng} with Lemma \ref{lem:perturbation}.

The following result, which bounds the convergence time of the dual function, is proved via induction. 
\begin{lemma}\label{lem:dual-convergence}
Suppose the assumptions in Theorem \ref{main-thm-2} hold. 
Let $\l\{\lambda^{(k)}\r\}_{k=0}^{K}$ be generated from Algorithm 2. 
For any $k = 0,1,2,\cdots,K$, we have
\[
F(\lambda^{(k)}) - F^*\leq \varepsilon_k + \varepsilon,
\]
where $\varepsilon_k = \varepsilon_0/2^k$.
\end{lemma}
\begin{proof}[Proof of Lemma \ref{lem:dual-convergence}]
First of all, for $k=0$, we have $\lambda^{(0)} = 0$ and
\[
F(\lambda^{(0)}) = - \max_{\mf x\in \mc X} f(\mf x)\leq M,
\]
thus, $F(\lambda^{(0)}) - F^*\leq 2M\leq \varepsilon_0+\varepsilon$, by the assumption that $2M\leq \varepsilon_0$ in Theorem \ref{main-thm-2}. Now for any $k>0$, let $\lambda_\varepsilon^{(k-1)}\in\mc S_{\varepsilon}$ be the closest point to $\lambda^{(k-1)}$ specified in Algorithm 2, i.e. $\lambda_\varepsilon^{(k-1)} = \text{argmin}_{\lambda_\varepsilon\in\mc S_\varepsilon} \|\lambda_\varepsilon - \lambda^{(k-1)}\|$. 
Suppose the claim holds for $(k-1)$-th stage, where $k >0$, then, consider the $k$-th stage.
\begin{enumerate}
\item If $F(\lambda^{(k-1)}) - F^*\leq \varepsilon$, then, $\lambda^{(k-1)}\in\mc S_{\varepsilon}$, thus, $\|\lambda_\varepsilon^{(k-1)}-\lambda^{(k-1)}\| = 0$. By \eqref{eq:dual-bound-2} with $\widetilde{\lambda} = \lambda^{(k-1)}$ from Algorithm 2 and $\lambda$ chosen to be $\lambda_\varepsilon^{(k-1)}$, we have
\[
F(\lambda^{(k)}) - F(\lambda_\varepsilon^{(k-1)})
\leq \frac{D^2}{2}\mu_k\leq \frac{\varepsilon_k}{2},
\]
Thus, it follows, $F(\lambda^{(k)}) - F^*
=  F(\lambda^{(k)}) - F(\lambda_\varepsilon^{(k-1)}) + F(\lambda_\varepsilon^{(k-1)})- F^* \leq \varepsilon_k +\varepsilon.$

\item If $F(\lambda^{(k-1)}) - F^*> \varepsilon$, then, $\lambda^{(k-1)}\not\in\mc S_{\varepsilon}$ and we claim that 
\begin{equation}\label{eq:attain-on-boundary}
F(\lambda^{(k-1)}_{\varepsilon}) - F^* = \varepsilon.
\end{equation}
Indeed, suppose on the contrary, $F(\lambda^{(k-1)}_{\varepsilon}) - F^* < \varepsilon$, then, by the continuity of the function $F$, 
there exists $\alpha\in(0,1)$ and $\lambda'=\alpha\lambda^{(k-1)}_{\varepsilon}+(1-\alpha)\lambda^{(k-1)}$ such that $F(\lambda') - F^* = \varepsilon$, i.e. $\lambda'\in\mc S_{\varepsilon}$, and $\|\lambda^{(k-1)}-\lambda'\| 
= \alpha\|\lambda^{(k-1)}-\lambda^{(k-1)}_{\varepsilon}\| < \|\lambda^{(k-1)}-\lambda^{(k-1)}_{\varepsilon}\| $, contradicting the fact that $\lambda_\varepsilon^{(k-1)} = \text{argmin}_{\lambda_\varepsilon\in\mc S_\varepsilon} \|\lambda_\varepsilon - \lambda^{(k-1)}\|$.

On the other hand, by induction hypothesis, we have
\[
F(\lambda^{(k-1)}) - F^*\leq \varepsilon_{k-1} + \varepsilon,
\]
which, combining with \eqref{eq:attain-on-boundary},  implies $F(\lambda^{(k-1)}) - F(\lambda_\varepsilon^{(k-1)})\leq \varepsilon_{k-1}$,
and by Lemma \ref{lem:yang-lin},
\begin{multline*}
\|\lambda^{(k-1)} - \lambda^{(k-1)}_\varepsilon\|\leq \frac{\text{dist}(\lambda_\varepsilon^{(k-1)}, \Lambda^*)}{\varepsilon}
\l( F(\lambda^{(k-1)}) - F(\lambda^{(k-1)}_\varepsilon) \r)\\
\leq\frac{C_\delta  \l(F(\lambda_\varepsilon^{(k-1)}) - F^*\r)^\beta\l( F(\lambda^{(k-1)}) - F(\lambda_\varepsilon^{(k-1)}) \r) }{\varepsilon}
\leq\frac{C_\delta \varepsilon_{k-1}}{\varepsilon^{1-\beta}},
\end{multline*}
where the second inequality follows from $\varepsilon\leq \delta$ assumed in Theorem \ref{main-thm-2} and the local error bound condition \eqref{eq:local-error-bound}. 
Note that by definition of $\theta_t$ in Algorithm 1, 
$\frac{1}{\theta_{T-1}^2}\geq T^2
\geq\frac{4D^2C_\delta^2\sigma_{\max}(\mf A^T\mf A)(2M)^{\beta}}{\varepsilon^{4/(2+\beta)}}$, and $\mu_k =  \varepsilon_k /D^2$.
Substituting these quantities into \eqref{eq:dual-bound-2} with $\widetilde{\lambda} = \lambda^{(k-1)}$ and $\lambda$ chosen to be $\lambda_\varepsilon^{(k-1)}$, we have
\begin{align*}
F(\lambda^{(k)}) - F(\lambda_\varepsilon^{(k-1)})\leq& \frac{D^2}{2}\mu_k + \theta_{T-1}^2\frac{\sigma_{\max}(\mf A^T\mf A)\|\lambda^{(k-1)}_\varepsilon - \lambda^{(k-1)}\|^2}{\mu_k} \\
\leq&\frac{\varepsilon_k}{2}
+\frac{\varepsilon^{4/(2+\beta)}}{2(2M)^{\beta}\varepsilon^{2(1-\beta)}}\varepsilon_k
=\frac{\varepsilon_k}{2}
+\frac{\varepsilon^{\frac{2\beta(1+\beta)}{2+\beta}}}{2(2M)^{\beta}}\varepsilon_k\\
\leq&\frac{\varepsilon_k}{2}\l(1+ \l(\frac{\varepsilon}{2M}\r)^\beta\r)\leq \varepsilon_k,
\end{align*}
where the second from the last inequality follows from $\varepsilon\leq1$ and the last inequality follows from  $\varepsilon\leq 2M$ assumed in Theorem \ref{main-thm-2}. Thus, it follows $F(\lambda^{(k)}) - F^* \leq \varepsilon_k +\varepsilon.$
\end{enumerate}
Overall, we finish the proof.
\end{proof}

\begin{proof}[Proof of Theorem \ref{main-thm-2}]
Since the desired accuracy is chosen small enough so that $\varepsilon\leq \frac\delta2$, and the number of stages $K\geq \lceil\log_2(\varepsilon_0/\varepsilon)\rceil + 1$, it follows $\varepsilon_{K-1}\leq \varepsilon\leq \frac\delta2$, and thus there exists some threshold $k'\in\{0,1,2,\cdots,K-1\}$ such that for any $k\geq k'$, $\varepsilon_k + \varepsilon \leq  \delta$. As a consequence, by Lemma \ref{lem:dual-convergence}, we have for any $k\geq k'$,
\[
F(\lambda^{(k)}) - F^*\leq \varepsilon_k + \varepsilon \leq\delta,
\]
i.e. $\lambda^{(k)}\in\mc S_\delta$, the $\delta$-sublevel set of the function $F(\lambda)$. By the local error bound condition \eqref{eq:local-error-bound}, we have
\[
\text{dist}(\lambda^{(k)},\Lambda^*)\leq \l(F(\lambda^{(k)}) - F^*\r)^\beta\leq(\varepsilon_k + \varepsilon )^\beta.
\]
Now, consider the $(k+1)$-th stage in the homotopy method. By \eqref{eq:constraint} in Theorem \ref{main-thm-1}, 
\begin{multline}\label{eq:inter-constraint}
\|\mf A\overline{\mf x}^{(k+1)}-\mf b\| \leq \frac{2\sigma_{\max}(\mf A^T\mf A)}{\mu_{k+1} S_T}\l(\|\lambda^{(k)}_*-\mf \lambda^{(k)}\| + \text{dist}(\lambda_{\mu_{k+1}}^*,\Lambda^*)\r)\\
\leq \frac{2\sigma_{\max}(\mf A^T\mf A)}{\mu_{k+1} S_T}\l( (\varepsilon_k + \varepsilon )^\beta + \text{dist}(\lambda_{\mu_{k+1}}^*,\Lambda^*)\r),
\end{multline}
where $\lambda^{(k)}_* = \text{argmin}_{\lambda^*\in\Lambda^*}\|\lambda^* - \lambda^{(k)}\|$, and the second inequality follows from 
\begin{equation}\label{eq:inter-step-bound}
\|\lambda^{(k)}_*-\mf \lambda^{(k)}\|= \text{dist}(\lambda^{(k)},\Lambda^*)\leq (\varepsilon_k + \varepsilon )^\beta.
\end{equation}
To bound the second term on the right hand side of \eqref{eq:inter-constraint}, note that $\mu_{k+1} = \varepsilon_{k+1}/D^2= \varepsilon_{k}/(2D^2) \leq\delta/(2D^2)$. Thus, by Lemma \ref{lem:perturbation},
\begin{multline*}
F(\lambda_{\mu_{k+1}}^*) - F(\lambda^*) = F(\lambda_{\mu_{k+1}}^*) - F_{\mu_{k+1}}(\lambda_{\mu_{k+1}}^*)  + F_{\mu_{k+1}}(\lambda_{\mu_{k+1}}^*) - F(\lambda^*)\\
\leq \frac{\mu_{k+1}}{2}D^2 + 0 = \mu_{k+1}D^2/2\leq \delta/2,
\end{multline*}
thus, it follows $\lambda_{\mu_{k+1}}^*\in\mc S_{\delta}$ and by local error bound condition
\[
\text{dist}(  \lambda_{\mu_{k+1}}^*,\Lambda^* )\leq C_\delta\l(F(\lambda_{\mu_{k+1}}^*) - F(\lambda^*) \r)^\beta\leq C_\delta\l(\varepsilon_k+\varepsilon\r)^\beta.
\]
Overall, substituting this bound into \eqref{eq:inter-constraint} ,we get
\begin{multline*}
\| \mf A\overline{\mf x}^{(k+1)}-\mf b \|
\leq  \frac{2\sigma_{\max}(\mf A^T\mf A)}{\mu_{k+1} S_T}\l( 1+C_\delta \r)\l(\varepsilon_k+\varepsilon\r)^\beta
\leq  \frac{4\sigma_{\max}(\mf A^T\mf A) D^2}{\varepsilon_{k+1} T^2}\l( 1+C_\delta \r)\l(\varepsilon_k+\varepsilon\r)^\beta,
\end{multline*}
where we use the fact that $\mu_{k+1} = \varepsilon_{k+1}/D^2$ and $S_T = \sum_{t=0}^{T-1}\frac{1}{\theta_t}\geq \sum_{t=1}^{T}t\geq \frac{T^2}{2}$.
Substituting the bound $T^2\geq\frac{4D^2C_\delta^2\sigma_{\max}(\mf A^T\mf A)(2M)^{\beta}}{\varepsilon^{4/(2+\beta)}}$ gives for any $k\geq k'$,
\begin{multline}\label{eq:constraint-bound}
\| \mf A\overline{\mf x}^{(k+1)}-\mf b \| \leq \frac{1+C_\delta}{C_\delta^2(2M)^\beta}\frac{(\varepsilon_k+\varepsilon)^\beta\varepsilon^{4/(2+\beta)}}{\varepsilon_{k+1}}
=\frac{2(1+C_\delta)}{C_\delta^2(2M)^\beta}\frac{(\varepsilon_k+\varepsilon)^\beta\varepsilon^{4/(2+\beta)}}{\varepsilon_{k}}\\
\leq \frac{2(1+C_\delta)}{C_\delta^2(2M)^\beta}\frac{(3\varepsilon_k)^\beta(4\varepsilon_k)^{4/(2+\beta)}}{\varepsilon_{k}}
\leq \frac{24(1+C_\delta)}{C_\delta^2(2M)^\beta}\varepsilon_k^{1+\frac{\beta^2}{2+\beta}},
\end{multline}
where the equality follows from $\varepsilon_{k+1} = \varepsilon_k/2$, and the second inequality follows from $\varepsilon \leq 2 \varepsilon_k,~\forall k\in\{0,1,2,\cdots,K-1\}$.
For the objective bound, we have by \eqref{eq:objective}, for any $k\geq k'$,
\begin{align}
f( \overline{\mf x}^{(k+1)} ) - f^*
\leq& \|\mc P_{\mf A}\lambda_0^*\|\cdot\|\mf A\overline{\mf x}^{(k+1)} - \mf b\| 
+ \frac{\sigma_{\max}(\mf A^T\mf A)}{2\mu_{k+1} S_T}\|\lambda^{(k)}_*-\lambda^{(k)}\|^2 + \frac{\mu_{k+1} D^2}{2} \nonumber \\
\leq& \|\mc P_{\mf A}\lambda_0^*\|\frac{24(1+C_\delta)}{C_\delta^2(2M)^\beta}\varepsilon_k^{1+\frac{\beta^2}{2+\beta}}
+ \frac{\sigma_{\max}(\mf A^T\mf A)}{2\mu_{k+1} S_T}\|\lambda^{(k)}_*-\lambda^{(k)}\|^2 + \frac{\mu_{k+1} D^2}{2}, \label{eq:inter-objective}
\end{align}
where the second inequality follows from \eqref{eq:constraint-bound}.
Now, for the second term on the right hand side, we have
\begin{multline*}
 \frac{\sigma_{\max}(\mf A^T\mf A)}{2\mu_{k+1} S_T}\|\lambda^{(k)}_*-\lambda^{(k)}\|^2
 \leq \frac{\varepsilon^{4/(2+\beta)}(\varepsilon_k+\varepsilon)^{2\beta}}{4\varepsilon_{k+1}C_\delta^2(2M)^\beta}
 = \frac{\varepsilon^{4/(2+\beta)}(\varepsilon_k+\varepsilon)^{2\beta}}{2\varepsilon_{k}C_\delta^2(2M)^\beta}\\
 \leq\frac{(4\varepsilon_k)^{4/(2+\beta)}(3\varepsilon_k)^{2\beta}}{2\varepsilon_{k}C_\delta^2(2M)^\beta}
 \leq\frac{6\varepsilon_k^{1+\frac{2\beta(1+\beta)}{2+\beta}}}{C_\delta^2(2M)^\beta},
\end{multline*}
where first inequality follows from \eqref{eq:inter-step-bound},
the equality follows from $\varepsilon_{k+1} = \varepsilon_k/2$, and the second inequality follows from $\varepsilon \leq 2 \varepsilon_k,~\forall k\in\{0,1,2,\cdots,K-1\}$. Substituting this bound and $\mu_{k+1}= \varepsilon_{k+1}/D^2 = \varepsilon_{k}/2D^2$ into \eqref{eq:inter-objective} gives for any $k\geq k'$,
\begin{equation}\label{eq:objective-bound}
f( \overline{\mf x}^{(k+1)} ) - f^*
\leq  \frac{24\|\mc P_{\mf A}\lambda_0^*\|(1+C_\delta)}{C_\delta^2(2M)^\beta}\varepsilon_k^{1+\frac{\beta^2}{2+\beta}}+\frac{6}{C_\delta^2(2M)^\beta}\varepsilon_k^{1+\frac{2\beta(1+\beta)}{2+\beta}}+\frac14\varepsilon_k.
\end{equation}
Taking $k = K-1$ in \eqref{eq:constraint-bound} and \eqref{eq:objective-bound} with the fact that $\varepsilon_{K-1}\leq \varepsilon\leq1$ gives the desired result.
\end{proof}

\subsection{Proof of Lemma \ref{lem:sol-sub}}\label{proof-explicit}
\begin{proof}
For simplicity of notations, we let $\mf x_i^* = \mf x_i(\widehat{\lambda}_t)$. 
First of all, let $H_C(\mf x_i)$ be the indicator function for the set $C:=\l\{\mf x_i:~\|\mf x_i - \mf b_i\|\leq D\r\}$, which takes 0 if $\mf x_i\in C$ and $+\infty$ otherwise.  Then, the optimization problem \eqref{eq:dec-sub} can be equivalently written as an unconstrained problem:
\begin{equation}\label{eq:dec-sub-2}
\mf x_i^* = \text{argmax}_{\mf x_i\in\mb R^d}-\frac{\mu}{2}\l\| \mf x_i - \mf a_i \r\|^2 - \|\mf x_i - \mf b_i\| - H_C(\mf x_i) =: g(\mf x_i),
\end{equation}
where $\mf a_i = \widetilde{\mf x}_i-\frac{1}{\mu}\sum_{j\in\mc N_i}\mf W_{ji}\lambda_{t,j}$. Since $\mf x_i^*$ is the solution, by the optimality condition, 
$0\in\partial g\l(\mf x_i^*\r)$, where $\partial g(\mf x_i^*)$ denotes the set of subdifferentials of $g$ at point  $\mf x_i^*$, i.e.
\[
0\in  \mu\l( \mf x_i^* - \mf a_i \r) + \partial \|\mf x_i^* - b_i\| + \mathcal{N}_C(\mf x_i^*),
\]
 where for any $\mf x\in\mb R^d$,
$$\partial \|\mf x - \mf b_i\| = 
\begin{cases}
\l\{\frac{\mf x - \mf b_i}{\|\mf x - \mf b_i\|}\r\},~~&\text{if}~~\mf x\neq \mf b_i,\\
\l\{\mf v\in\mb R^d, \|\mf v\|\leq 1\r\},~~&\text{otherwise},
\end{cases}
$$ 
and $\mathcal{N}_C(\mf x)$ is the normal cone of the set $C=\l\{\mf x_i:~\|\mf x_i - \mf b_i\|\leq D\r\}$ at the point $\mf x$, i.e.
\[
\mathcal{N}_C(\mf x) := \l\{ \mf v\in \mb R^d:~\mf v^T\mf x\geq \mf v^T\mf y,~\forall \mf y\in C \r\}.
\]
This is equivalent to
\begin{equation}\label{eq:sol-condition}
-\mu\l( \mf x_i^* - \mf a_i \r) - \mf h\in \mathcal{N}_C(\mf x_i^*),
\end{equation}
for some $\mf h\in\partial \|\mf x_i^* - \mf b_i\|$. Note that the function $g(\cdot)$ is a strongly concave function, thus, the solution to the maximization problem \eqref{eq:dec-sub-2} is unique, which implies as long as one can find one $x_i^*$ and $\mf h$ satisfying \eqref{eq:sol-condition}, such a $x_i^*$ must be the only solution.  To this point, we consider the following three cases:
\begin{enumerate}
\item If $\|\mf b_i - \mf a_i\|\leq 1/\mu$. Let $\mf x_i^* = \mf b_i$ and $\mf h = \mu(\mf a_i  - \mf b_i)$, then, $\mathcal{N}_C(\mf x_i^*)=\l\{0\r\}$ and $\|\mf h\|\leq 1 $ and $-\mu\l( \mf x_i^* - \mf a_i \r) - \mf h = 0\in \mathcal{N}_C(\mf x_i^*)$.
\item If $1/\mu<\|\mf b_i - \mf a_i\|\leq 1/\mu+ D$, then, one can take 
$$\mf x_i^* = \mf b_i - \frac{\mf b_i - \mf a_i}{\|\mf b_i - \mf a_i\|}\l(\|\mf b_i - \mf a_i\| - \frac{1}{\mu}\r)=\mf a_i +\frac{\mf b_i - \mf a_i}{\|\mf b_i - \mf a_i\|}\frac{1}{\mu}$$ 
and 
$\mf h = \frac{\mf a_i - \mf b_i}{\|\mf a_i-\mf b_i\|}$. Note that $\|\mf x_i^* - \mf b_i\| = \|\mf a_i-\mf b_i\| - 1/\mu \leq D$, 
which again gives $\mathcal{N}_C(\mf x_i^*)=\l\{0\r\}$ and $-\mu\l( \mf x_i^* - \mf a_i \r) - \mf h = 0 \in \mathcal{N}_C(\mf x_i^*)$.
\item If $\|\mf b_i - \mf a_i\|> 1/\mu+ D$. Then, let $\mf x_i^* = \mf b_i - \frac{\mf b_i - \mf a_i}{\|\mf b_i - \mf a_i\|}D$ and $\mf h = \frac{\mf a_i - \mf b_i}{\|\mf a_i - \mf b_i\|}$, which gives
\begin{align*}
-\mu\l( \mf x_i^* - \mf a_i \r) - \mf h =& -\mu\l( \mf b_i - \mf a_i - \frac{\mf b_i - \mf a_i}{\|\mf b_i - \mf a_i\|}D \r) - \frac{\mf a_i - \mf b_i}{\|\mf a_i - \mf b_i\|}\\
=& -\mu(\mf b_i - \mf a_i)\l( 1- \frac{D}{\|\mf b_i - \mf a_i\|} \r) - \frac{\mf a_i - \mf b_i}{\|\mf a_i - \mf b_i\|}\\
=& -\mu(\mf b_i - \mf a_i)\l( 1- \frac{D+1/\mu}{\|\mf b_i - \mf a_i\|} \r) = \mu\l( 1- \frac{D+1/\mu}{\|\mf b_i - \mf a_i\|} \r)(\mf a_i - \mf b_i). \\
\end{align*}
Note that the normal $\mc N_C(\mf x_i^*)  = \l\{ c(\mf a_i-\mf b_i),c\geq0 \r\}$, it follows $-\mu\l( \mf x_i^* - \mf a_i \r) - \mf h\in \mc N_C(\mf x_i^*)$.
\end{enumerate}
Overall, we finish the proof.
\end{proof}

\subsection{Proof of Theorem \ref{thm:geo-local-bound}}\label{proof-local-geo}
Since the null space of $\mf A$ is non-empty and the set 
$$\mathcal{X}:=\l\{\mf x\in\mb R^{nd}: ~\|\mf x_i - \mf b_i\|\leq D,i=1,2,\cdots,n \r\}$$ 
is compact, strong duality holds with respect to (\ref{dec-prob-1}-\ref{dec-prob-2}). In view of Assumption \ref{assumption}(c)(d), we aim to show that the Lagrange dual of (\ref{dec-prob-1}-\ref{dec-prob-2}) satisfies the local error bound condition \eqref{eq:local-error-bound} and the set of optimal Lagrange multiplier is unique up to null space of $\mf A$.

We start by rewriting (\ref{dec-prob-1}-\ref{dec-prob-2}) as follows: Let $\mf y_i = \mf x_i - \mf b_i$, and $\mf y = [\mf y_1^T,~\mf y_2^T,\cdots,~\mf y_n^T]^T$, then, (\ref{dec-prob-1}-\ref{dec-prob-2}) is equivalent to
\begin{align*}
\min&~~\sum_{i=1}^n\|\mf y_i\|  \\
s.t.&~~ \mf A\mf y + \mf A\mf b = 0, \|\mf y_i\|\leq D,~i=1,2,\cdots,n.
\end{align*}
Then, for any $\lambda\in\mb R^{nd}$, the Lagrange dual function
\begin{align*}
F(\lambda) =&  \max_{\|\mf y_i\|\leq D,~i=1,2,\cdots,n}-\sum_{i=1}^n\|\mf y_i\| - \dotp{\lambda}{\mf A\mf y + \mf A\mf b }\\
=& \underbrace{ \max_{\|\mf y_i\|\leq D,~i=1,2,\cdots,n}-\sum_{i=1}^n\l(\|\mf y_i\| + \dotp{\lambda}{\mf A_{[i]}\mf y_i }\r)}_{\text{(I)}}  
- \dotp{\lambda}{\mf A\mf b},
\end{align*}
where 
$$\mf A_{[i]}= [\mf W_{1i}~\mf W_{2i}~\cdots~\mf W_{ni}]^T
$$ 
$i$-th column block of the matrix $\mf A$ corresponding to $\mf y_i$. Note that maximization of (I) is separable with respect to the index $i$, we have for any $i\in\{1,2,\cdots,n\}$, 
\begin{multline*}
\max_{\|\mf y_i\|\leq D} -\|\mf y_i\| - \dotp{\lambda}{\mf A_{[i]}\mf y_i } = 
\max_{\|\mf y_i\|\leq D} -\|\mf y_i\| - \dotp{\mf A_{[i]}^T\lambda}{\mf y_i }\\
=
\begin{cases}
0,&~~\text{if}~\|\mf A_{[i]}^T\lambda\|\leq1\\
(\|\mf{A}_{[i]}^T\lambda\|-1)\cdot D,&~~\text{otherwise}.
\end{cases}
\end{multline*}
Thus, one can write $F(\lambda)$ as follows 
\begin{equation}\label{eq:dual-geo-median}
F(\lambda) = - \dotp{\mf A^T\lambda}{\mf b} + D\sum_{i=1}^n(\|\mf{A}_{[i]}^T\lambda\|-1)\cdot I\l(\|\mf{A}_{[i]}^T\lambda\|>1\r),
\end{equation}
where $I\l(\|\mf{A}_{[i]}^T\lambda\|>1\r)$ is the indicator function which takes 1 if $\|\mf{A}_{[i]}^T\lambda\|>1$ and 0 otherwise. To this point, we make another change of variables by setting $\nu_i = \mf{A}_{[i]}^T\lambda,~i=1,2,\cdots,n$ and $\nu = [\nu_1^T~\nu_2^T~\cdots~\nu_n^T]^T$. Note that $\{\mf A^T\lambda:\lambda\in\mb R^{nd}\} = \mc R(\mf A^T)$.
By the null space property \eqref{eq:null-space}, the range space of $\mf A^T$ has the following explicit representation:
\begin{equation}\label{eq:range-space}
\mc R(\mf A^T) = \l\{ \mf \nu\in\mb R^{nd}:~\mf u = [\nu_1^T,\cdots,\nu_n^T]^T,~\sum_{i=1}^n\nu_i = 0\r\}.
\end{equation}
Thus, minimizing \eqref{eq:dual-geo-median} is equivalent to solving the following constrained optimization problem:
\begin{align}
\min_{\nu\in\mb R^{nd}}&~~ - \dotp{\nu}{\mf b} + D\sum_{i=1}^n(\|\nu_i\|-1)\cdot I\l(\|\nu_i\|>1\r),\label{change-prob-1}\\
s.t.&~~\sum_{i=1}^n\nu_i = 0,\label{change-prob-2}
\end{align}
Denote 
\begin{equation}\label{eq:def-G}
G(\nu) = - \dotp{\nu}{\mf b} + D\sum_{i=1}^n(\|\nu_i\|-1)\cdot I\l(\|\nu_i\|>1\r).
\end{equation}
The following lemma, which characterizes the set of solutions to (\ref{change-prob-1}-\ref{change-prob-2}), paves the way of our analysis.

\begin{lemma}\label{lem:solution-set-1}
The solution to (\ref{change-prob-1}-\ref{change-prob-2}) is attained within the region: $\mc B = \{\nu\in\mb R^{nd}, \|\nu_i\|\leq1,~\forall i\}$. Furthermore, for any $\nu'\in\mb R^{nd}$ satisfying \eqref{change-prob-2} but not in $\mc B$, there exists a point $\overline{\nu}'\in\mc B$ such that \eqref{change-prob-2} is satisfied and
\[
G(\nu') - G(\overline{\nu}')\geq \l(\max_{i,j}\|\mf b_i - \mf b_j\|\r) \|\nu' - \overline{\nu}'\|.
\]
\end{lemma}
\begin{proof}[Proof of Lemma \ref{lem:solution-set-1}]
Consider any $\nu'\in\mb R^{nd}$ not in the set $\mc B$, then, 
define the set $\mathcal{J}$ as
the set of coordinates $j$ in $\{1, ...., n\}$ such that $\|\nu'_j\|>1$.   Since $\nu'$
is not in the set $\mathcal{B}$, we know $\mathcal{J}$ is nonempty. 
Then, let $L := \max_{j\in\mc J}\|\nu_j'\|>1$. Consider the vector $\overline{\nu}' := \nu'/L$, then, since $\nu'$ is a solution to (\ref{change-prob-1}-\ref{change-prob-2}), $\sum_{i=1}^n\nu_i' = 0$, which implies $\sum_{i=1}^n\overline{\nu}_i'  = 0$. Furthermore, we obviously have $\|\overline{\nu}_i'\|\leq1,~\forall i$. Now, we are going to show that $G(\nu')>G(\overline{\nu}')$, thereby reaching a contradiction. Consider the difference
\begin{align*}
&G(\nu') - G(\overline{\nu}')\\
 =& \dotp{\overline{\nu}'-\nu'}{\mf b} + D\sum_{i=1}^n(\|\nu_i'\|-1)\cdot I\l(\|\nu_i'\|>1\r)\\
=& \sum_{i=1}^{n-1}\dotp{\overline{\nu}'_i-\nu'_i}{\mf b_i - \mf b_n}+ D\sum_{i=1}^n(\|\nu_i'\|-1)\cdot I\l(\|\nu_i'\|>1\r)
~~\l(\text{by the fact}~\sum_{i=1}^n\overline{\nu}_i'= \sum_{i=1}^n\nu_i' = 0 \r) \\
\geq& \sum_{i=1}^{n-1}\dotp{\overline{\nu}'_i-\nu'_i}{\mf b_i - \mf b_n} + (L-1)D\\
\geq& -\sum_{i=1}^{n-1}\|\overline{\nu}'_i-\nu'_i\| \cdot \|\mf b_i - \mf b_n\| + (L-1)D~~(\text{by Cauchy-Schwarz})\\
\geq& - \l(\max_{i,j}\|\mf b_i - \mf b_j\|\r)\sum_{i=1}^{n-1}\|\overline{\nu}'_i-\nu'_i\|  + (L-1)D\\
=& - \l(\max_{i,j}\|\mf b_i - \mf b_j\|\r)\sum_{i=1}^{n-1}\|\overline{\nu}'_i\|(L-1)  + (L-1)D~~\l(\text{By definition}~\overline{\nu}' := \nu'/L\r)\\
\geq& - \l(\max_{i,j}\|\mf b_i - \mf b_j\|\r)\sum_{i=1}^{n-1}\|\overline{\nu}'_i\|(L-1)+ 2(L-1)\cdot n\cdot \max_{i,j}\|\mf b_i - \mf b_j\|~~
\l(D\geq 2 n\cdot \max_{i,j}\|\mf b_i - \mf b_j\|\r)\\
\geq&\l(\max_{i,j}\|\mf b_i - \mf b_j\|\r)\sum_{i=1}^{n}\|\overline{\nu}'_i\|(L-1)~~(\text{by the fact}~\|\overline{\nu}'_i\|\leq1)\\
\geq&\l(\max_{i,j}\|\mf b_i - \mf b_j\|\r)\sum_{i=1}^{n}\|\nu'_i - \overline{\nu}'_i\|
\geq \l(\max_{i,j}\|\mf b_i - \mf b_j\|\r)\|\nu' - \overline{\nu}'\|,~~\l(\text{By definition}~\overline{\nu}' := \nu'/L\r)
\end{align*}
and the lemma follows.
\end{proof}

By the previous lemma, in order to characterize the set of solutions to (\ref{change-prob-1}-\ref{change-prob-2}), it is enough to look at the following more restricted problem:
\begin{align}
\min_{\nu\in\mb R^{nd}}&~~ - \dotp{\nu}{\mf b},\label{another-prob-1}\\
s.t.&~~\sum_{i=1}^n\nu_i = 0,\label{another-prob-2}\\
&~~\|\nu_i\|^2\leq1,~i=1,2,\cdots,n,   \label{another-prob-3}
\end{align}
where we used the fact that $G(\nu) = -\dotp{\nu}{\mf b}$ when $\|\nu_i\|^2\leq1,~\forall i$. This is a quadratic constrained problem. Now, we show the key lemma that $G(\nu)$ satisfies the local error bound with parameter $\beta=1/2$ over the restricted set \eqref{another-prob-2} and \eqref{another-prob-3}.

\begin{lemma}\label{lem:local-error-G}
The solution to (\ref{another-prob-1}-\ref{another-prob-3}) is \textit{unique}. Furthermore,
let $\nu^*\in\mb R^{nd}$ be the solution to (\ref{another-prob-1}-\ref{another-prob-3}). 
There exists a constant $C_0>0$ such that for any $\nu\in\mb R^{nd}$ satisfying (\ref{another-prob-2}-\ref{another-prob-3}),
\[
\|\nu - \nu^*\|\leq C_0\l(G(\nu) - G(\nu^*)\r)^{1/2}.
\]
\end{lemma}

The proof of Lemma \ref{lem:local-error-G} is somewhat lengthy, but it follows a simple intuition that if the solution point lies on the boundary of a ball, then, sliding a point away from the solution results in a locally quadratic growth of the objective when it is linear. We split the proof into two cases below.

\subsubsection{Proof of Lemma \ref{lem:local-error-G}: Case 1}

\textbf{Case 1:} The solution of the original geometric median (\ref{dec-prob-1}-\ref{dec-prob-2}) is achieved at one of the vectors $\{\mf b_1,~\mf b_2,~\cdots,~\mf b_n\}$.

Assume without loss of generality that it is achieved at $\mf x_1 = \mf x_2 = \cdots = \mf x_n= \mf b_n$, then, one know that the minimum of (\ref{dec-prob-1}-\ref{dec-prob-2}) is 
$\sum_{i=1}^{n-1}\|\mf b_i - \mf b_n\|$. Furthermore, since we assume $\{\mf b_1,~\mf b_2,~\cdots,~\mf b_n\}$ is not co-linear, the solution is unique, and thus, for all feasible $\mf x\neq [\mf b_n^T,~\mf b_n^T,~\cdots,~\mf b_n^T]^T$, $\sum_{i=1}^{n}\|\mf x_i - \mf b_i\|>\sum_{i=1}^{n-1}\|\mf b_n - \mf b_i\|$.

First, one can get rid of constraint \eqref{another-prob-2} in (\ref{another-prob-1}-\ref{another-prob-3}) by substituting $\nu_n = -\sum_{i=1}^{n-1}\nu_i$ and equivalently form the following optimization problem:
\begin{align}
\min_{\nu\in\mb R^{nd}}&~~ - \sum_{i=1}^{n-1}\dotp{\nu_i}{\mf b_i-\mf b_n},\label{another-prob-1'}\\
s.t.
&~~\|\nu_i\|^2\leq1,~i=1,2,\cdots,n-1,   \label{another-prob-2'}\\
&~~\l\| \sum_{i=1}^{n-1}\nu_i \r\|\leq1. \label{another-prob-3'}
\end{align}
Then, to show the uniqueness of the solution to (\ref{another-prob-1}-\ref{another-prob-3}), it is enough to show the solution to (\ref{another-prob-1'}-\ref{another-prob-3'}) is unique. To see the the uniqueness, suppose we temporarily delete constraint \eqref{another-prob-3'}, then we obtain a relaxed problem:
\begin{align*}
\min_{\nu\in\mb R^{nd}}&~~ - \sum_{i=1}^{n-1}\dotp{\nu_i}{\mf b_i-\mf b_n},\\
s.t.
&~~\|\nu_i\|^2\leq1,~i=1,2,\cdots,n-1,  
\end{align*}
which is separable and we know trivially that for each index $i$, the solution to 
\[
\min_{\nu_i\in\mb R^{d}} - \dotp{\nu_i}{\mf b_i-\mf b_n},~~s.t.~~\|\nu_i\|^2\leq1,
\]
is attained \textit{uniquely} at $\nu_i^* = \frac{\mf b_i-\mf b_n}{\|\mf b_i-\mf b_n\|}$. This gives the objective value $-\sum_{i=1}^{n-1}\|\mf b_n - \mf b_i\|$ to the relaxed problem. On the other hand, by strong duality, the optimal objective of the original problem (\ref{another-prob-1}-\ref{another-prob-3}) is also $-\sum_{i=1}^{n-1}\|\mf b_n - \mf b_i\|$.
The fact that the optimal objective does not change even when adding an extra constraint $\l\| \sum_{i=1}^{n-1}\nu_i \r\|\leq1$ implies that
$\nu_i^* = \frac{\mf b_i-\mf b_n}{\|\mf b_i-\mf b_n\|},~i=1,2,\cdots,n-1$ is feasible with respect to (\ref{another-prob-1}-\ref{another-prob-3}), and
 the solution to (\ref{another-prob-1}-\ref{another-prob-3}) cannot be attained at any feasible point other than  $\nu_i^* = \frac{\mf b_i-\mf b_n}{\|\mf b_i-\mf b_n\|},~i=1,2,\cdots,n-1$. As a consequence, the solution to (\ref{another-prob-1}-\ref{another-prob-3}) is also unique, which is 
 $\nu_i^* =  \frac{\mf b_i-\mf b_n}{\|\mf b_i-\mf b_n\|},~i=1,2,\cdots,n-1$ and $\nu_n^* = -\sum_{i=1}^{n-1}\nu_i$.
 
 \begin{figure}[htbp]
 \centering
   \includegraphics[width=4in]{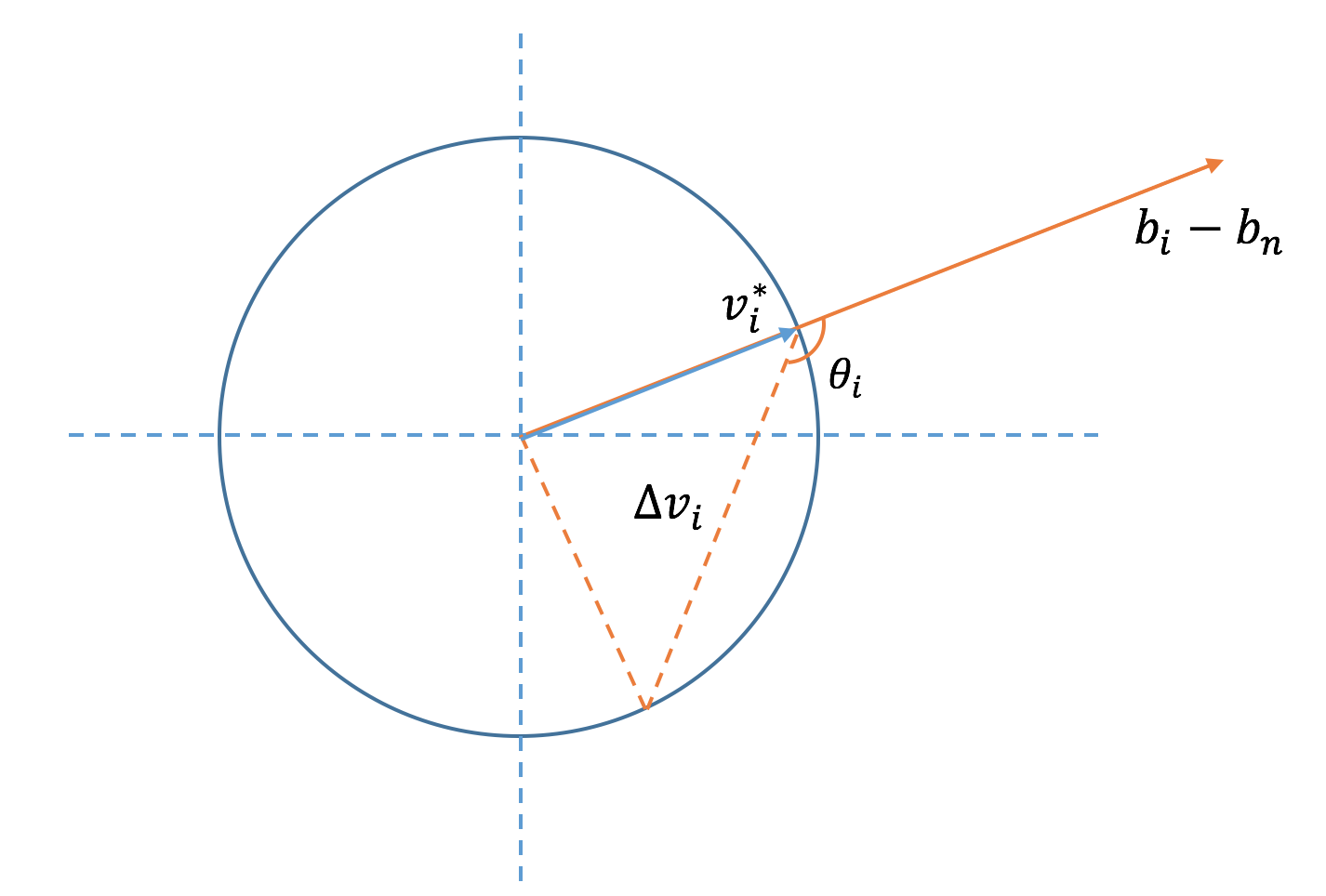} 
   \caption{Geometric interpretation of the local perturbation by $\Delta \nu_i$ around the solution $\nu_i^*$. For any perturbation $\Delta \nu_i$ of fixed length, the maximum of 
   $\dotp{\mf b_i - \mf b_n}{\Delta\nu_i}$ is achieved when $\|\nu^* +\Delta\nu_i\| = 1$, i.e. $\nu^* +\Delta\nu_i$ is on the boundary of the unit ball, in which case we have $\cos\theta_i =- \|\Delta \nu_i\|/2$ and $\dotp{\mf b_i - \mf b_n}{\Delta\nu_i} =  \|\mf b_i - \mf b_n\| \cdot\| \Delta\nu_i \|\cos\theta_i=- \|\mf b_i - \mf b_n\| \cdot\| \Delta\nu_i \|^2/2$.} 
   \label{fig:geometry}
\end{figure}

Next, we are going to show a local error bound condition for (\ref{another-prob-1'}-\ref{another-prob-3'}), and then pass the result back to (\ref{another-prob-1}-\ref{another-prob-3}). 
To this point, we consider any perturbation $\Delta \nu = [\Delta \nu_1^T,~\Delta \nu_2^T,~\cdots,~\Delta \nu_{n}^T]^T$ around the solution to 
(\ref{another-prob-1'}-\ref{another-prob-3'}) so that $\nu^* + \Delta\nu$ is within the feasible set $\l\{\nu\in\mb R^{nd}:~\|\nu_i\|^2\leq1,~i=1,2,\cdots,n-1,\l\| \sum_{i=1}^{n-1}\nu_i \r\|\leq1.\r\}$.
It follows $\sum_{i=1}^n(\nu^*_i + \Delta\nu_i) = 0$, which implies $\Delta\nu_n = -\sum_{i=1}^{n-1}\Delta\nu_i$.
Furthermore,
$\|\nu_i^*+\Delta\nu_i\|\leq 1,~\forall i=1,2,\cdots,n-1$ and 
$\l\| \sum_{i=1}^{n-1}(\nu_i^*+\Delta\nu_i) \r\|\leq1$. 

Denote $q(\nu) := - \sum_{i=1}^{n-1}\dotp{\nu_i}{\mf b_i-\mf b_n}$. Then, we have
\begin{equation}\label{eq:local-error-inter}
q(\nu^*+\Delta\nu) - q(\nu^*) = -\sum_{i=1}^{n-1}\dotp{\Delta \nu_i}{\mf b_i - \mf b_n}.
\end{equation}
Recall that $\|\nu_i^*+\Delta\nu_i\|\leq 1$ and $\nu_i^* = \frac{\mf b_i - \mf b_n}{\|\mf b_i - \mf b_n\|}$, it follows,
\[
\l\| \frac{\mf b_i - \mf b_n}{\|\mf b_i - \mf b_n\|} + \Delta\nu_i \r\|^2 \leq 1.
\]
Expanding the squares gives
\[
1+ 2\dotp{\frac{\mf b_i - \mf b_n}{\|\mf b_i - \mf b_n\|}}{\Delta\nu_i} + \| \Delta\nu_i \|^2\leq 1.
\]
Rearranging the terms gives
\[
\dotp{\mf b_i - \mf b_n}{\Delta\nu_i} \leq - \|\mf b_i - \mf b_n\| \cdot\| \Delta\nu_i \|^2/2.
\]
A geometric interpretation of this bound is given in Fig. \ref{fig:geometry}. 
Substituting this bound into \eqref{eq:local-error-inter} gives
\begin{align*}
q(\nu^*+\Delta\nu) - q(\nu^*)\geq& \sum_{i=1}^{n-1}  \|\mf b_i - \mf b_n\|\cdot\frac{ \|\Delta\nu_i\|^2}{2}\\
\geq&\frac12\l(\min_i \|\mf b_i - \mf b_n\|\r)\sum_{i=1}^{n-1}\|\Delta\nu_i\|^2.
\end{align*}
Note that since $\l\{ \mf b_1,~\mf b_2,~\cdots,~\mf b_n \r\}$ are distinct, $\min_i \|\mf b_i - \mf b_n\|>0$ and this
gives a local error bound condition for  (\ref{another-prob-1'}-\ref{another-prob-3'}) with parameter $\beta=\frac12$. Finally, since $\Delta\nu_n = -\sum_{i=1}^{n-1}\Delta\nu_i$, it follows,
\begin{multline*} 
q(\nu^*+\Delta\nu) - q(\nu^*)\geq\frac12\l(\min_i \|\mf b_i - \mf b_n\|\r)\sum_{i=1}^{n-1}\|\Delta\nu_i\|^2\\
\geq\frac{1}{2(n-1)}\l(\min_i \|\mf b_i - \mf b_n\|\r)\l\|\sum_{i=1}^{n-1}\Delta\nu_i\r\|^2
= \frac{1}{2(n-1)}\l(\min_i \|\mf b_i - \mf b_n\|\r)\l\|\Delta\nu_n\r\|^2,
\end{multline*}
where the second inequality follows from Cauchy-Schwarz inequality that
\[
\sqrt{\sum_{i=1}^{n-1}\|\Delta\nu_i\|^2}\sqrt{n-1}\geq \sum_{i=1}^{n-1}\|\Delta \nu_i\|\geq\l\|\sum_{i=1}^{n-1}\Delta\nu_i\r\|.
\]
Since $G(\nu+\Delta\nu) - G(\nu^*) = q(\nu^*+\Delta\nu) - q(\nu^*)$, it follows 
\[
G(\nu+\Delta\nu) - G(\nu^*)\geq  \frac{1}{4(n-1)}\l(\min_i \|\mf b_i - \mf b_n\|\r)\sum_{i=1}^{n}\|\Delta\nu_i\|^2
= \frac{1}{4(n-1)}\l(\min_i \|\mf b_i - \mf b_n\|\r)\|\Delta\nu\|^2.
\]
Finishing the proof for case 1.

\subsubsection{Proof of Lemma \ref{lem:local-error-G}: Case 2}

\textbf{Case 2:} The solution of the original geometric median (\ref{dec-prob-1}-\ref{dec-prob-2}) is NOT achieved at any of the vectors $\{\mf b_1,~\mf b_2,~\cdots,~\mf b_n\}$.

We start by rewriting problem (\ref{another-prob-1}-\ref{another-prob-3}) as an equivalent feasibility problem:
\begin{equation}\label{eq:feasibility}
\begin{cases}
 - \dotp{\nu}{\mf b} - G(\nu^*) \leq 0,\\
\|\nu_i\|^2\leq 1,~~i=1,2,\cdots,n,\\
\sum_{i=1}^n\nu_i = 0.
\end{cases}
\end{equation}

The uniqueness in this case comes from the following lemma.
\begin{lemma}\label{lem:feasible-unique}
The solution $\nu^*\in\mb R^{nd}$ to \eqref{eq:feasibility} is unique and satisfies $\|\nu_i^*\| = 1,~\forall i=1,2,\cdots,n$.
\end{lemma}

To understand the feasibility problem \eqref{eq:feasibility} and prove Lemma \ref{lem:feasible-unique}, we start with the following definition:
\begin{definition}[\cite{wang1994global}]
Consider any inequality system $f_i(\mf x)\leq 0,~i=1,2,\cdots,m$. An inequality $f_i(\mf x)\leq 0$ in the system is said to be singular if $f_i(\mf x) = 0$ for any solution to the system. If every inequality in the system is singular, we say the inequality system is singular.
\end{definition}

The following basic lemma regarding general feasibility problems is also proved in (\cite{wang1994global}).
\begin{lemma}[Lemma 2.1 of \cite{wang1994global}]\label{lem:Pang}
Consider any inequality system $f_i(\mf x)\leq 0,~i=1,2,\cdots,m$ with non-empty solution set $S$. Suppose each of $f_i$ is convex. Denote 
\begin{align*}
K &:= \l\{ k\in\{1,2,\cdots,m\}: f_k(\mf x) \leq 0~\text{is nonsingular} \r\},\\
J &:= \l\{ j\in\{1,2,\cdots,m\}: f_j(\mf x) \leq 0~\text{is singular} \r\}.
\end{align*}
Then, the sub-system $f_j(\mf x)\leq 0, j\in J$ alone is singular.
\end{lemma}

\begin{proof}[Proof of Lemma \ref{lem:feasible-unique}]
Suppose $\nu^*$ is one of the solutions to \eqref{eq:feasibility}. 
Suppose without loss of generality, the ball constraint $\|\nu_n\|^2\leq1$ in \eqref{eq:feasibility} is nonsingular. Then, by Lemma \ref{lem:Pang},  the subsystem 
\begin{equation}\label{eq:another-feasibility}
\begin{cases}
 - \dotp{\nu}{\mf b} - G(\nu^*) \leq 0,\\
\|\nu_i\|^2\leq 1,~~i=1,2,\cdots,n-1,\\
\sum_{i=1}^n\nu_i = 0.
\end{cases}
\end{equation}
is still singular. This implies the optimal objective value of the following problem 
\begin{align*}
\min_{\nu\in\mb R^{nd}}&~~ - \dotp{\nu}{\mf b},\\
s.t.&~~\sum_{i=1}^n\nu_i = 0,\\
&~~\|\nu_i\|^2\leq1,~i=1,2,\cdots,n-1,  
\end{align*}
is still $G(\nu^*)$. Similar as before, one can get rid of the equality using $\nu_n = -\sum_{i=1}^{n-1}\nu_i$ and form an equivalent problem:
\begin{align*}
\min_{\nu\in\mb R^{nd}}&~~ - \sum_{i=1}^{n-1}\dotp{\nu_i}{\mf b_i-\mf b_n},\\
s.t.
&~~\|\nu_i\|^2\leq1,~i=1,2,\cdots,n-1.  
\end{align*}
This is a separable problem and obviously the optimal objective of this problem is $-\sum_{i=1}^{n-1}\|\mf b_i - \mf b_n\|$, which implies 
$G(\nu^*) = -\sum_{i=1}^{n-1}\|\mf b_i - \mf b_n\|$. However, by strong duality and the uniqueness of the geometric median problem (\ref{dec-prob-1}-\ref{dec-prob-2}), this further implies the solution to (\ref{dec-prob-1}-\ref{dec-prob-2}) is attained uniquely at $\mf x_1 = \mf x_2=\cdots=\mf x_n = \mf b_n$, contradicting the assumption that the solution to (\ref{dec-prob-1}-\ref{dec-prob-2}) is NOT achieved at any of the vectors $\{\mf b_1,~\mf b_2,~\cdots,~\mf b_n\}$. Thus, we have shown that it is not possible to have one of the ball constraint being loose. This trivially implies it is not possible to have any two or more ball constraints being loose and hence we know that any solution $\nu^*$ to \eqref{eq:feasibility} must satisfy $\|\nu_i^*\| = 1,~\forall i=1,2,\cdots,n$.

Now suppose on the contrary such a solution is not unique. Let $\nu^*,~\wt{\nu}^*\in\mb R^{nd}$ be two distinct solutions. Then, they must be different at some index $j$, i.e. $\exists j$ such that $\nu_j^*\neq\wt{\nu}_j^*$ and they satisfy $\|\nu_j^*\|=\|\wt{\nu}_j^*\| = 1$ by the previous argument. However, since the solution set to \eqref{eq:feasibility} must be convex (which follows trivially from the fact that all constraints are convex), any convex combination of $\nu^*,~\wt{\nu}^*$ must be the solution. Specifically, the solution $\frac{\nu^*+\wt{\nu}^*}{2}$ has its $j$-th index
$\l\|\frac{\nu_j^*+\wt{\nu}_j^*}{2}\r\| <1$, contradicting the fact that any solution $\nu^*$ must satisfy $\|\nu_i^*\| = 1,~\forall i=1,2,\cdots,n$.
\end{proof}

Now, we proceed to prove Lemma \ref{lem:local-error-G} for this case. The proof is inspired by a crucial ``linearization'' technique transforming general quadratic systems to linear systems which we are able to understand (e.g. \cite{wang1994global}, \cite{luo1994extension}). Consider any feasible $\nu\in\mb R^{nd}$ regarding \eqref{another-prob-1}-\eqref{another-prob-3}. Then, for any index $i$, we have
\begin{multline}\label{eq:bound-i}
\|\nu_i - \nu_i^*\|^2  = \|\nu_i\|^2 - 2\dotp{\nu_i}{\nu_i^*} + \|\nu_i^*\|^2 = \|\nu_i\|^2 - 2\dotp{\nu_i- \nu_i^*}{\nu_i^* } - \|\nu_i^*\|^2\\
=\|\nu_i\|^2 - 1 + 2\dotp{\nu_i^*- \nu_i}{\nu_i^* }\leq2\dotp{\nu_i^*- \nu_i}{\nu_i^* },
\end{multline}
where in the third equality we use Lemma \ref{lem:feasible-unique} that $\|\nu_i^*\|=1$. We aim to bound the second term $\dotp{\nu_i^*- \nu_i}{\nu_i^* }$.

By Lemma \ref{lem:feasible-unique}, we have the following system has NO solution:
\begin{equation}\label{eq:nonlinear-inter}
\begin{cases}
 - \dotp{\nu}{\mf b} - G(\nu^*) \leq 0,\\
\|\nu_i\|^2- 1 < 0,~~i=1,2,\cdots,n,\\
\sum_{i=1}^n\nu_i = 0.
\end{cases}
\end{equation}
This is equivalent to claiming the following \textit{linear} system has no solution:
\begin{equation}\label{eq:linear-inter}
\begin{cases}
 - \dotp{\mf b}{\mf y} \leq 0,\\
\dotp{\nu^*_i}{\mf y_i} < 0,~~i=1,2,\cdots,n,\\
\sum_{i=1}^n\mf y_i = 0.
\end{cases}
\end{equation}
To see why this is true, suppose on the contrary, \eqref{eq:linear-inter} indeed has a solution. Let $\mf y^*$ be its solution, then we have $\alpha \mf y^*$ is also a solution for any $\alpha>0$. This in turn implies 
$$- \dotp{\mf b}{\nu^* + \alpha \mf y^*} - G(\nu^*) \leq - \dotp{\mf b}{\nu^*} - G(\nu^*)\leq0,$$
and
$$ \sum_{i=1}^n(\nu_i + \alpha\mf y_i^*) = \alpha \sum_{i=1}^n\mf y_i = 0. $$
Furthermore, for sufficiently small $\alpha$, e.g. we can choose any $\alpha \leq \min_i \frac{\dotp{\nu_i^*}{\mf y_i^*}}{\|\mf y_i^*\|^2}$, the following holds,
\[
\dotp{\nu_i^*}{\alpha \mf y_i^*} + \alpha^2\|\mf y_i^*\|^2\leq 0.
\]
This implies
\[
\|\nu_i^* + \alpha \mf y_i^*\| = \|\nu_i^*\|^2 + 2\dotp{\nu_i^*}{\alpha y_i^*} + \|\mf y_i^*\|-1\leq \dotp{\nu_i}{\alpha\mf y_i^*}<0,
\]
and thus $\nu_i^* + \alpha\mf y_i^*$ is a solution to \eqref{eq:nonlinear-inter}. 
On the other hand, suppose \eqref{eq:nonlinear-inter} has a solution, then, one can show similarly \eqref{eq:linear-inter} has a solution.

To analyze \eqref{eq:linear-inter}, we employ the classical Motzkin's alternative theorem:
\begin{lemma}[\cite{linear1952}, Theorem D6]
Suppose $\mf A \neq 0$. Either 
\[
\mf A\mf x > 0,~~\mf B\mf x \geq 0, ~~ \mf C\mf x = 0, 
\]
has a solution, or there exists $\mf u, \mf v, \mf w$ such that
\[
\mf A^T\mf u + \mf B^T\mf v + \mf C^T\mf w = 0, ~~\mf u \geq 0, ~~\mf v\geq0, \mf u \neq 0,
\]
but not both, where the inequalities are taken to be entrywise.
\end{lemma}

Now, applying Motzkin's alternative to \eqref{eq:linear-inter}, we have there exists a $\mf u\in\mb R^{2n+1}$ such that
\begin{equation}\label{eq:alternative}
-u_0\mf b^T  + \sum_{i=1}^n u _i \l[\nu_i^*\r] + \sum_{i=1}^nu_{n+i}[\mf e_i] = 0,~~[u_1,~u_2,~\cdots,~u_n]\neq0,~\mf u \geq0,
\end{equation}
where we define the block notation ``$[\cdot]$'' as follows 
$$
\l[\nu_i^*\r] = [\mf 0,\cdots,~\mf 0,~(\nu_i^*)^T,~\mf 0,\cdots,~\mf 0]\in\mb R^{nd},
$$
which takes $\nu_i^*$ at the $i$-th block of dimension $d$ and $\mf 0$ on other blocks. Also,
$$
[\mf e_i] = [\mf e_i^T,~\mf e_i^T,~\cdots,~\mf e_i^T]\in\mb R^{nd},
$$
which takes unit basis vector $\mf e_i\in\mb R^d$ on all blocks. 

\textbf{Claim 1:} $u_i>0,~\forall i=1,2,\cdots,n$.

To see why this is true, suppose on the contrary one of the $u_i$'s is $0$. Without loss of generality, we can assume $u_n = 0$. Then, by 
Motzkin's alternative again on \eqref{eq:alternative}, the following system has no solution:
\begin{equation}\label{eq:linear-inter-2}
\begin{cases}
 - \dotp{\mf b}{\mf y} \leq 0,\\
\dotp{\nu^*_i}{\mf y_i} < 0,~~i=1,2,\cdots,n-1,\\
\sum_{i=1}^n\mf y_i = 0.
\end{cases}
\end{equation}
By a similar equivalence relation as that of \eqref{eq:nonlinear-inter} and \eqref{eq:linear-inter}, this implies the following system has no solution,
\begin{equation*}
\begin{cases}
 - \dotp{\nu}{\mf b} - G(\nu^*) \leq 0,\\
\|\nu_i\|^2- 1 < 0,~~i=1,2,\cdots,n-1,\\
\sum_{i=1}^n\nu_i = 0,
\end{cases}
\end{equation*}
which, by substituting $\nu_n = -\sum_{i=1}^{n-1}\nu_{i}$, implies the following system has no solution:
\begin{equation}\label{eq:linear-inter-2}
\begin{cases}
 - \sum_{i=1}^{n-1}\dotp{\nu_i}{\mf b_i-\mf b_n} - G(\nu^*) \leq 0,\\
\|\nu_i\|^2- 1 < 0,~~i=1,2,\cdots,n-1.
\end{cases}
\end{equation}
However, we know that the solution to the following minimization problem:
\begin{align*}
\min_{\nu\in\mb R^{nd}}&~~ - \sum_{i=1}^{n-1}\dotp{\nu_i}{\mf b_i-\mf b_n},~~
s.t.
~~\|\nu_i\|^2\leq1,~i=1,2,\cdots,n-1,
\end{align*}
is attained uniquely at $\nu_i = \frac{\mf b_i-\mf b_n}{\|\mf b_i-\mf b_n\|}$ and the optimal objective value is $-\sum_{i=1}^{n-1}\|\mf b_i-\mf b_n\|$ which must be \textit{strictly less than} 
$G(\nu^*)$ by strong duality and the fact that the solution to (\ref{dec-prob-1}-\ref{dec-prob-2}) is not attained at $\mf x_1=\mf x_2=\cdots=\mf x_n = \mf b_n$. As a consequence, 
if we set 
$$\wt{\nu}_i = \frac{\mf b_i-\mf b_n}{\|\mf b_i-\mf b_n\|}\frac{-G(\nu^*)}{\sum_{i=1}^{n-1}\|\mf b_i-\mf b_n\|},~i=1,2,\cdots,n-1,$$
then, $\|\wt{\nu}_i\| < 1,~\forall i=1,2,\cdots,n-1$ and $ - \sum_{i=1}^{n-1}\dotp{\wt{\nu}_i}{\mf b_i-\mf b_n} - G(\nu^*) = 0$, which implies \eqref{eq:linear-inter-2} has a solution and we reach a contradiction.

Now, rewriting \eqref{eq:alternative}, we have
\[
[u_1(\nu_1^*)^T,~u_2(\nu_2^*)^T,~\cdots,~u_n(\nu_n^*)^T] = u_0 \mf b -  \sum_{i=1}^nu_{n+i}[\mf e_i],
\]
multiplying both sides by $[\nu_1^* - \nu_1,~\nu_2^* - \nu_2,~\cdots,~\nu_n^* - \nu_n]$,
which implies 
\begin{multline*}
\sum_{j=1}^n u_j \dotp{\nu_j^* - \nu_j}{\nu_j^*} = u_0\sum_{j=1}^n\dotp{\mf b_j}{\nu_j^* - \nu_j} - \sum_{i=1}^n\sum_{j=1}^nu_{n+i}\dotp{\mf e_i}{\nu_j^* - \nu_j}\\
=  u_0\sum_{j=1}^n\dotp{\mf b_j}{\nu_j^* - \nu_j} = u_0 (G(\nu) - G(\nu^*)),
\end{multline*}
where the second from the last equality follows from $\sum_{i=1}^n\nu_i = \sum_{i=1}^n\nu_i^* = 0$. Thus, for any index $j\in\{1,2,\cdots,n\}$, 
\begin{align*}
\dotp{\nu_j^* - \nu_j}{\nu_j^*}  =& \sum_{i\neq j} \frac{u_i}{u_j} \dotp{\nu_i - \nu_i^*}{\nu_i^*} + \frac{u_0}{u_j} (G(\nu) - G(\nu^*)) ~~~ (\text{by the fact}~u_j>0)\\
\leq& \sum_{i\neq j} \frac{u_i}{u_j} (\|\nu_i\|^2 - \|\nu_i^*\|^2) + \frac{u_0}{u_j} (G(\nu) - G(\nu^*))~~~(\text{by convexity and}~u_i>0)\\
\leq& \frac{u_0}{u_j} (G(\nu) - G(\nu^*))~~~(\text{by feasibility that} \|\nu_i\|^2\leq 1 = \|\nu_i^*\|^2).
\end{align*}
Substituting this bound into \eqref{eq:bound-i} gives
\[
\|\nu_j^* - \nu_j\|^2\leq \frac{2u_0}{u_j} (G(\nu) - G(\nu^*)),~\forall j\in\{1,2,\cdots,n\},
\]
and thus,
\[
\|\nu^* - \nu\|^2 = \sum_{j=1}^n\|\nu_j^* - \nu_j\|^2\leq \sum_j\frac{2u_0}{u_j} (G(\nu) - G(\nu^*)),
\]
finishing the proof.

\subsubsection{Putting everything together}
Combining Lemma \ref{lem:solution-set-1} and Lemma \ref{lem:local-error-G} we can easily show the following:
\begin{lemma}\label{lem:all-G}
The solution $\nu^*$ to (\ref{change-prob-1}-\ref{change-prob-2}) is unique and furthermore, for any $\delta>0$ and any point $\nu = [\nu_1^T,~\nu_2^T,~\cdots,~\nu_n^T]^T\in \mb R^{nd}$  such that 
$\sum_{i=1}^n\nu_i = 0$ and $G(\nu) - G(\nu^*)\leq \delta$, we have there exists a constant $C_\delta$ depending on $\delta$ such that
\[
G(\nu) - G(\nu^*)\geq C_\delta\|\nu - \nu^*\|^2.
\]
\end{lemma}
\begin{proof}[Proof of Lemma \ref{lem:all-G}]
Since the solution to (\ref{change-prob-1}-\ref{change-prob-2}) is attained in the constraint set (\ref{another-prob-2}-\ref{another-prob-3}) by Lemma  \ref{lem:solution-set-1}, the uniqueness follows  directly from Lemma \ref{lem:local-error-G}.

Now, for any $\nu\in\mb R^{nd}$, such that $\sum_{i=1}^n\nu_i=0$, and $\|\nu_i\| > 1$ for some index $i$,
\[
G(\nu) - G(\nu^*) = G(\nu) - G(\overline{\nu}) + G(\overline{\nu}) - G(\nu^*)
\geq\l(\max_{i,j}\|\mf b_i - \mf b_j\|\r)\| \nu - \overline{\nu} \| + C_0^2\|\overline{\nu} - \nu^*\|^2.
\]
where the vector $\overline{\nu}$ is defined in Lemma \ref{lem:solution-set-1}, the second inequality follows from Lemma \ref{lem:solution-set-1} that 
$G(\nu) - G(\overline{\nu})\geq\l(\max_{i,j}\|\mf b_i - \mf b_j\|\r)\| \nu - \overline{\nu} \|$ and Lemma \ref{lem:local-error-G} that 
$ G(\overline{\nu}) - G(\nu^*)\geq C_0^2\|\overline{\nu} - \nu^*\|^2$.

Thus, for any $\nu$ such that $G(\nu) - G(\nu^*)\leq \delta$, we have
\[
\frac{\delta}{\max_{i,j}\|\mf b_i - \mf b_j\|}\geq \| \nu - \overline{\nu} \|, 
\]
which implies
\[
\| \nu - \overline{\nu} \|\geq 
\begin{cases}
\frac{\max_{i,j}\|\mf b_i - \mf b_j\|}{\delta}\| \nu - \overline{\nu} \|^2,~~&\text{if}~\frac{\delta}{\max_{i,j}\|\mf b_i - \mf b_j\|}>1,\\
\| \nu - \overline{\nu} \|^2,~~&\text{otherwise}.
\end{cases}
\]
Thus, 
\begin{align*}
G(\nu) - G(\nu^*)\geq& C_0^2\|\overline{\nu} - \nu^*\|^2 + \frac{\max_{i,j}\|\mf b_i - \mf b_j\|}{\max\l\{ \frac{\delta}{\max_{i,j}\|\mf b_i - \mf b_j\|},1 \r\}}\| \nu - \overline{\nu} \|^2\\
\geq& C_\delta ( \|\overline{\nu} - \nu^*\| + \| \nu - \overline{\nu} \|)^2 \geq C_\delta\|\nu - \nu^*\|^2,
\end{align*}
for some $C_\delta >0$, where the second inequality follows from $\|w+z\|^2 \leq 2\|w\|^2 + 2\|z\|^2,~\forall w,z$ and the last inequality follows from triangle inequality.

On the other hand,  for any $\nu\in\mb R^{nd}$, such that $\sum_{i=1}^n\nu_i=0$, and $\|\nu_i\| \leq 1$ for all indices $i$, by Lemma \ref{lem:local-error-G}
\[
G(\nu) - G(\nu^*)\geq C_0^2\|\nu - \nu^*\|^2.
\]
Overall, we finish the proof.
\end{proof}

\subsubsection{Finishing the proof of Theorem \ref{thm:geo-local-bound}}
We recall the following well-known Hoffman's error bound:
\begin{lemma}[Theorem 9 of \cite{Pang1997}]
Given a convex polyhedron expressed as the solution set of a system of linear inequalities and equations defined by a pair of matrices $(\mf A,\mf B)$:
\[
S:=\l\{ \mf x\in\mb R^d:~\mf A\mf x \leq a,~\mf B\mf x = \mf b \r\}.
\]
There exists a scalar $c>0$ such that for all $(\mf a, \mf b)$ for which $S$ is non-empty,
\[
\text{dist}(\mf x, S) \leq c\l(\|(\mf A\mf x - \mf a)_+\| + \|\mf B\mf x - \mf b\|\r), ~\forall \mf x\in\mb R^d, 
\]
where for any vector $\mf y\in\mb R^n$, $\|(\mf y)_+\|:= \sqrt{\sum_{i=1}^n\max\{y_i,0\}^2}$.
\end{lemma}

The idea is to translate the local error bound on function $G(\nu)$ (i.e. Lemma \ref{lem:all-G}) back to the local error bound on the original dual function $F(\lambda)$ using the equivalence relation between minimizing the dual function \eqref{eq:dual-geo-median} and problem (\ref{change-prob-1}-\ref{change-prob-2}). Recall the definition of $F(\lambda)$ in \eqref{eq:dual-geo-median} and $G(\nu)$ in \eqref{eq:def-G}, we have $F(\lambda) = G(\nu)$ for any $\lambda\in\mb R^{nd}$ such that $\mf A^T\lambda = \nu$.
Thus, by Lemma \ref{lem:all-G}, with $\nu$ replaced by  $\mf A^T\lambda$ and  $G(\nu)$ replaced by $F(\lambda)$,
\[
\|\mf A^T\lambda - \nu^*\|\leq C_0
(F(\lambda) - F^*)^{1/2}, 
\]
where $F^*$ is the optimal dual function value, and we use the fact that $F^*$ equals $G(\nu^*)$, the optimal objective of  (\ref{change-prob-1}-\ref{change-prob-2}). Since the solution $\nu^*$ to (\ref{change-prob-1}-\ref{change-prob-2}) is unique, the set of optimal Lagrange multipliers (i.e. the set of minimizers of \eqref{eq:dual-geo-median}) $\Lambda^* = \l\{\lambda\in\mb R^{nd}:~\mf A^T\lambda = \nu^*\r\}$. By Hoffman's bound with $S = \Lambda^*$, we have
\[
\text{dist}(\lambda,\Lambda^*)\leq c\|\mf A^T\lambda- \nu^*\| 
\]
for some positive constant $c$. Thus,
\[
\text{dist}(\lambda,\Lambda^*)\leq\frac{C_0}{c}(F(\lambda) - F^*)^{1/2}.
\]
Furthermore, since for any $\lambda^*\in\Lambda^*$ there exists a unique $\nu^*$ such that $\mf A^T\lambda^* = \nu^*$, it follows $\mc P_{\mf A}\lambda^* = \mf A(\mf A^T\mf A)^{\dagger}\mf A^T\lambda^* = \mf A(\mf A^T\mf A)^{\dagger}\nu^*$.

\subsection{Simulation setups and additional simulation results}\label{add-simulation}
In this section, we give more details about our simulation along with more simulation results. First of all, in all three cases of Section \ref{sec:simulation}, the randomly generated graph are connected. The way we ensure its connectivity is to first connect all nodes together by assigning $(n-1)$ edges, and then, randomly pick the remaining edges from the edge set of $n(n+1)/2$ edges according to the connectivity ratio. An example graph containing 20 nodes with connectivity ratio of $0.13$ is shown in Fig. \ref{ss}.

 \begin{figure}[htbp]
 \centering
   \includegraphics[width=4in]{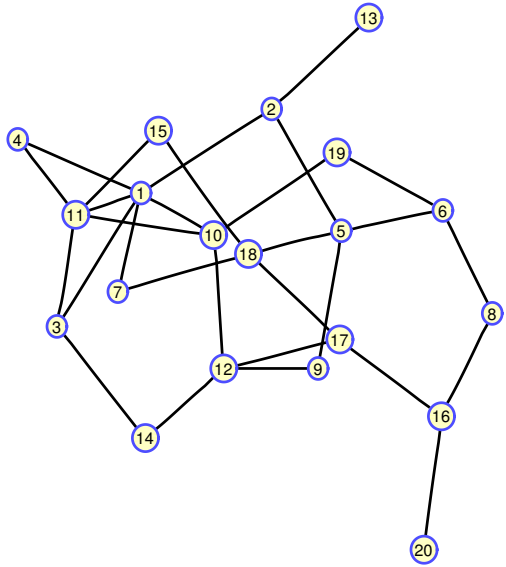} 
   \caption{Illustration of a randomly generated connected graph with $n =20$ and connectivity ratio=0.13.} 
   \label{ss}
\end{figure}

The parameters of algorithms are set as follows: (1) For the DSM algorithm, the learning rate $\alpha=10$. (2) For the EXTRA algorithm, the learning rate $\alpha=5$ when $n=20$ and $\alpha=20$ when $n=50,100$. (3) For the Jacobian ADMM, the proximal weight $\rho = 2\sigma_{\max}(\mf A)$, where $\sigma_{\max}(\mf A)$ is the maximum eigenvalue of $\mf A$.
(4) For the smoothing algorithm, we fix the smoothing parameter $\mu = 10^{-5}$ throughout the experiments. (5) For our proposed algorithm, we set $D = 10\sqrt{d}$, where $d$ is the dimension of the data and the desired accuracy
$\varepsilon=10^{-3}$. During the 
$k$-th stage, the time horizon $T^{(k)} = \frac{D}{\varepsilon^{0.8}}\cdot\frac{k}{K}$, where $K=\lceil\log_2(1/\varepsilon)\rceil+1$ is the total number of rounds. The reason why we consider increasing the time horizon gradually is that we observe in practice the algorithm converges very fast during the first few stages and it is not necessary to run a long time. The aforementioned parameters of all algorithms are chosen in an ad-hoc way to ensure good performances.

Here, we perform additional simulations to show that our algorithm also works well under other scenarios where we change the dimension of the data. In the experiment below, the number of agents is set to be $n=100$ and all the parameters are as described above. We vary the dimension of the data from 20 to 200, where each entry of the data points is still uniformly distributed over $[0,10]$. The results are shown in Fig. \ref{ss2}.
 
\begin{figure*}[ht!] 
    \centering
    \begin{subfigure}[t]{0.28\textwidth}
        \centering
        \includegraphics[height=3cm] {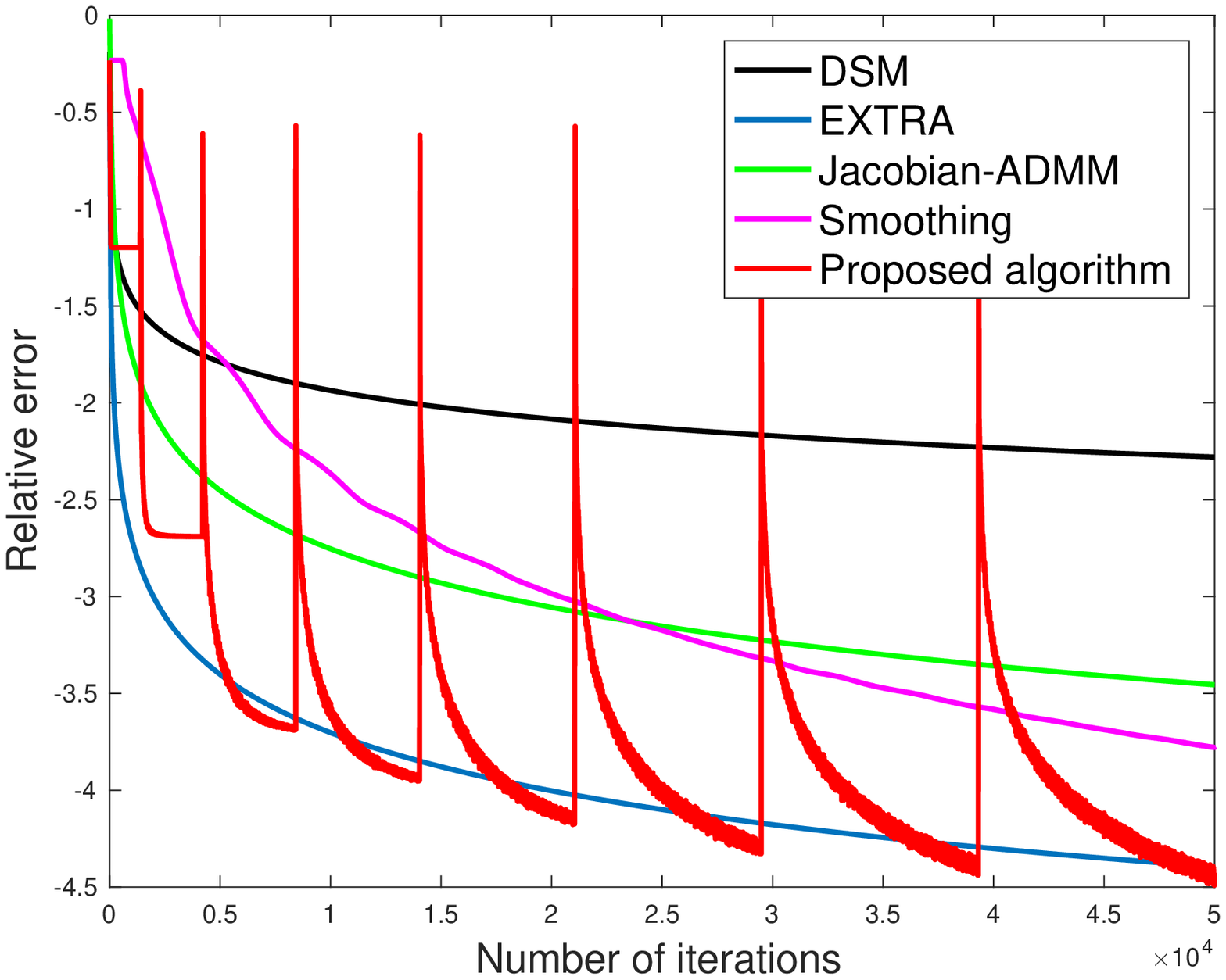}
        \vspace{-1.5em}
        \caption{$d=20$, ratio=0.15.}
    \end{subfigure}%
    ~ 
    \begin{subfigure}[t]{0.28\textwidth}
        \centering
        \includegraphics[height=3cm] {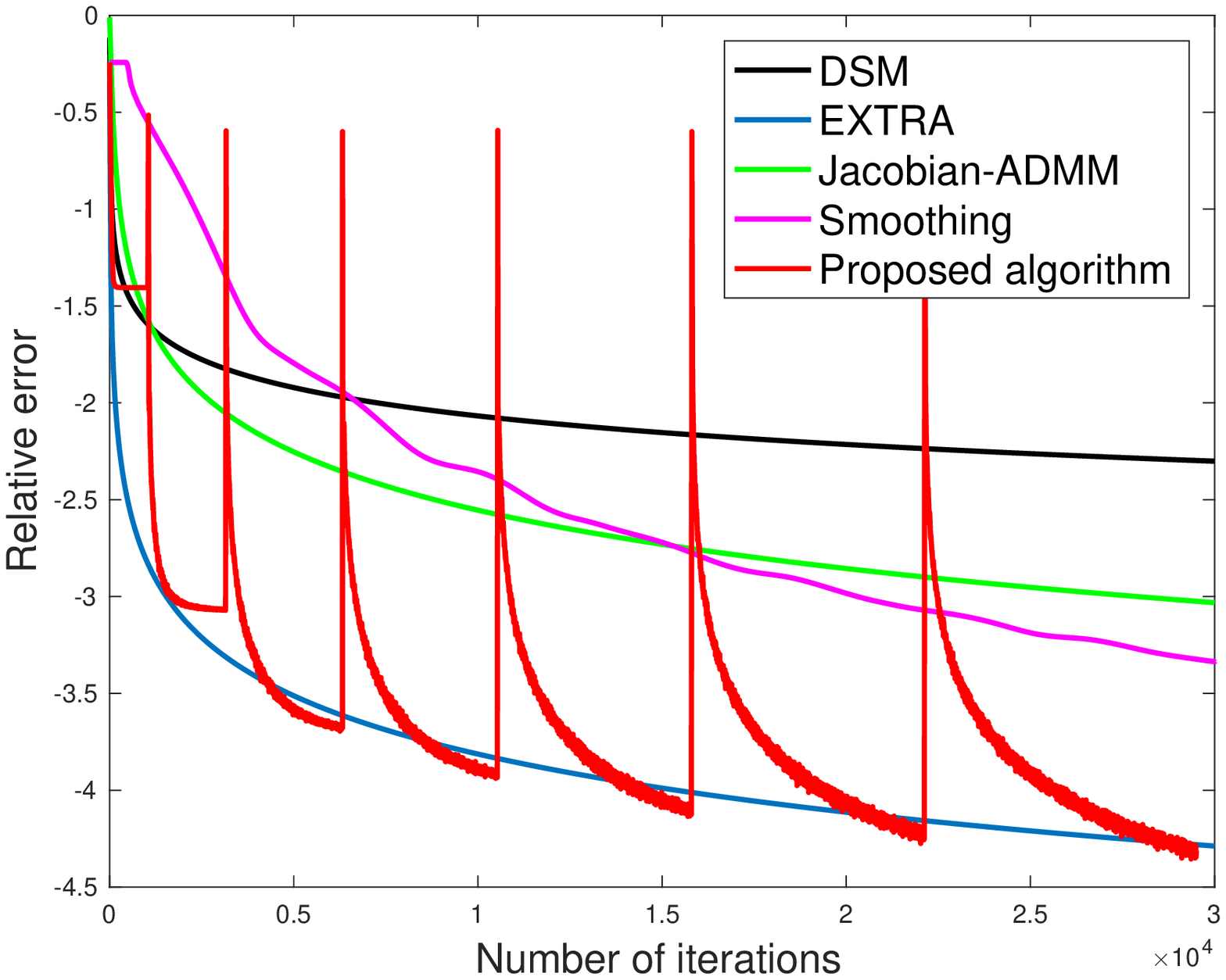}
        \vspace{-1.5em}
        \caption{$d=50$, ratio=0.1.}
    \end{subfigure}
    ~
    \begin{subfigure}[t]{0.28\textwidth}
        \centering
        \includegraphics[height=3cm] {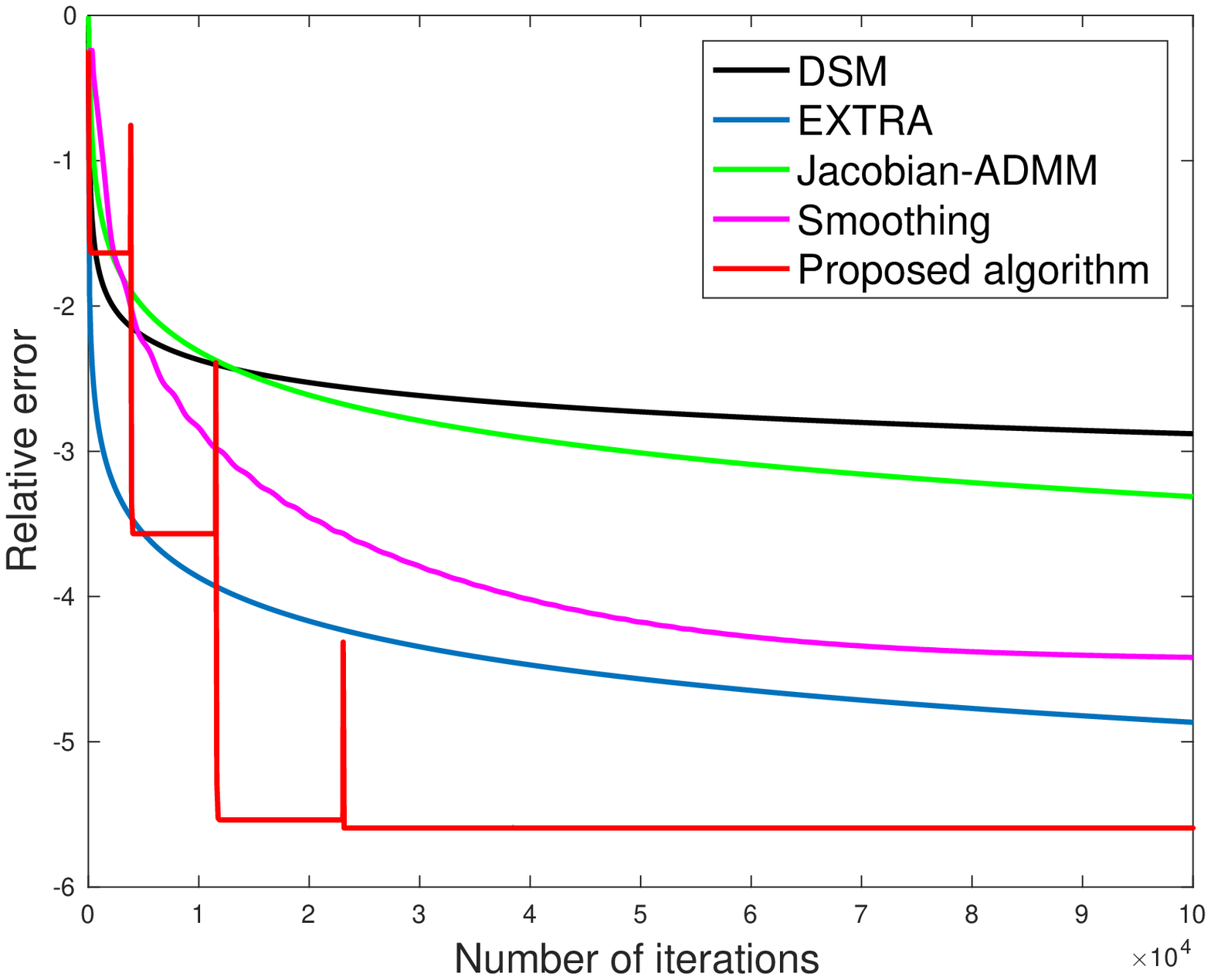}
        \vspace{-1.5em}
        \caption{$d=150$, ratio=0.1.}
    \end{subfigure}
        ~
     \begin{subfigure}[t]{0.28\textwidth}
        \centering
        \includegraphics[height=3cm] {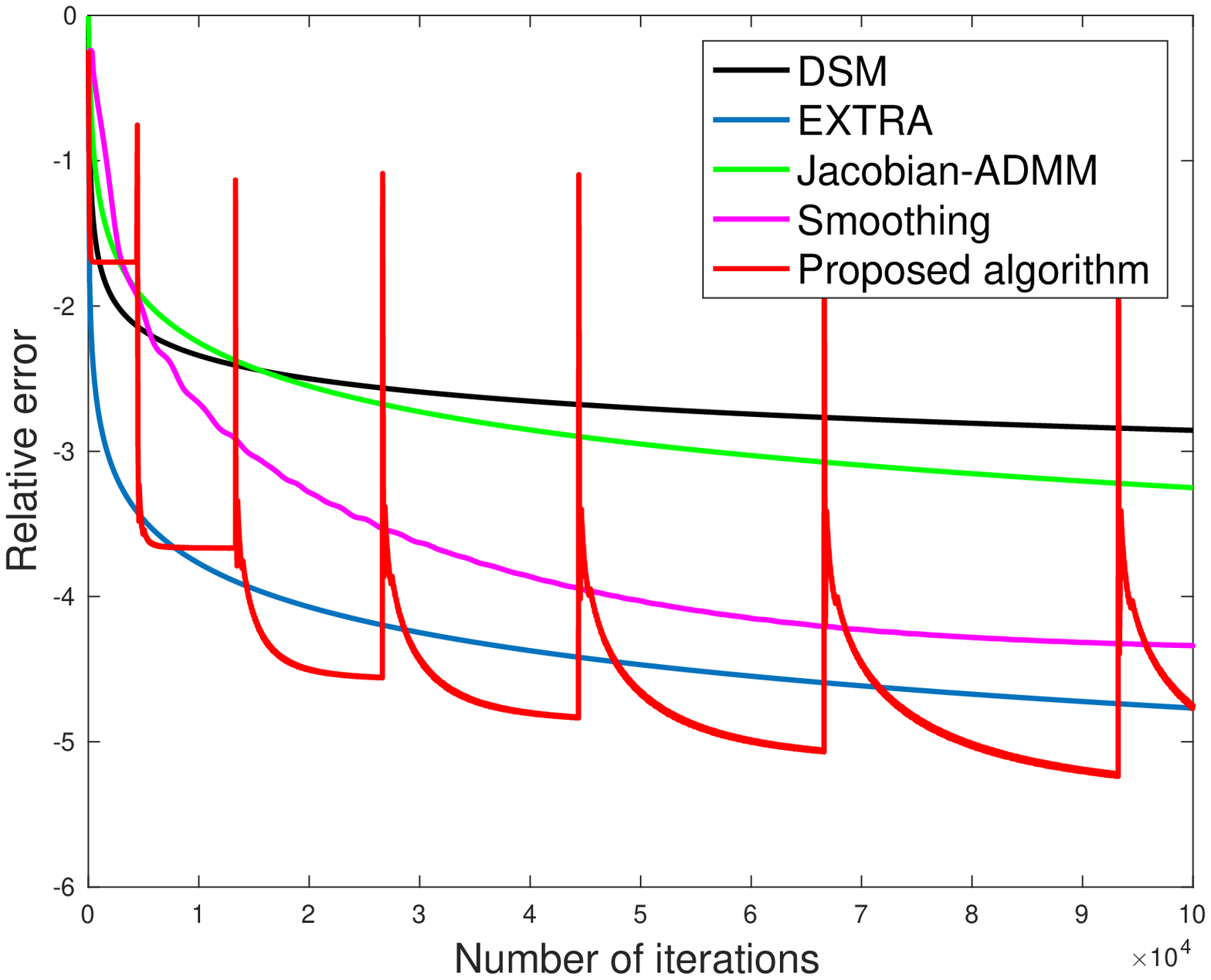}
        \vspace{-1.5em}
        \caption{$d=200$, ratio=0.1.}
    \end{subfigure}
    ~
     \begin{subfigure}[t]{0.28\textwidth}
        \centering
        \includegraphics[height=3cm] {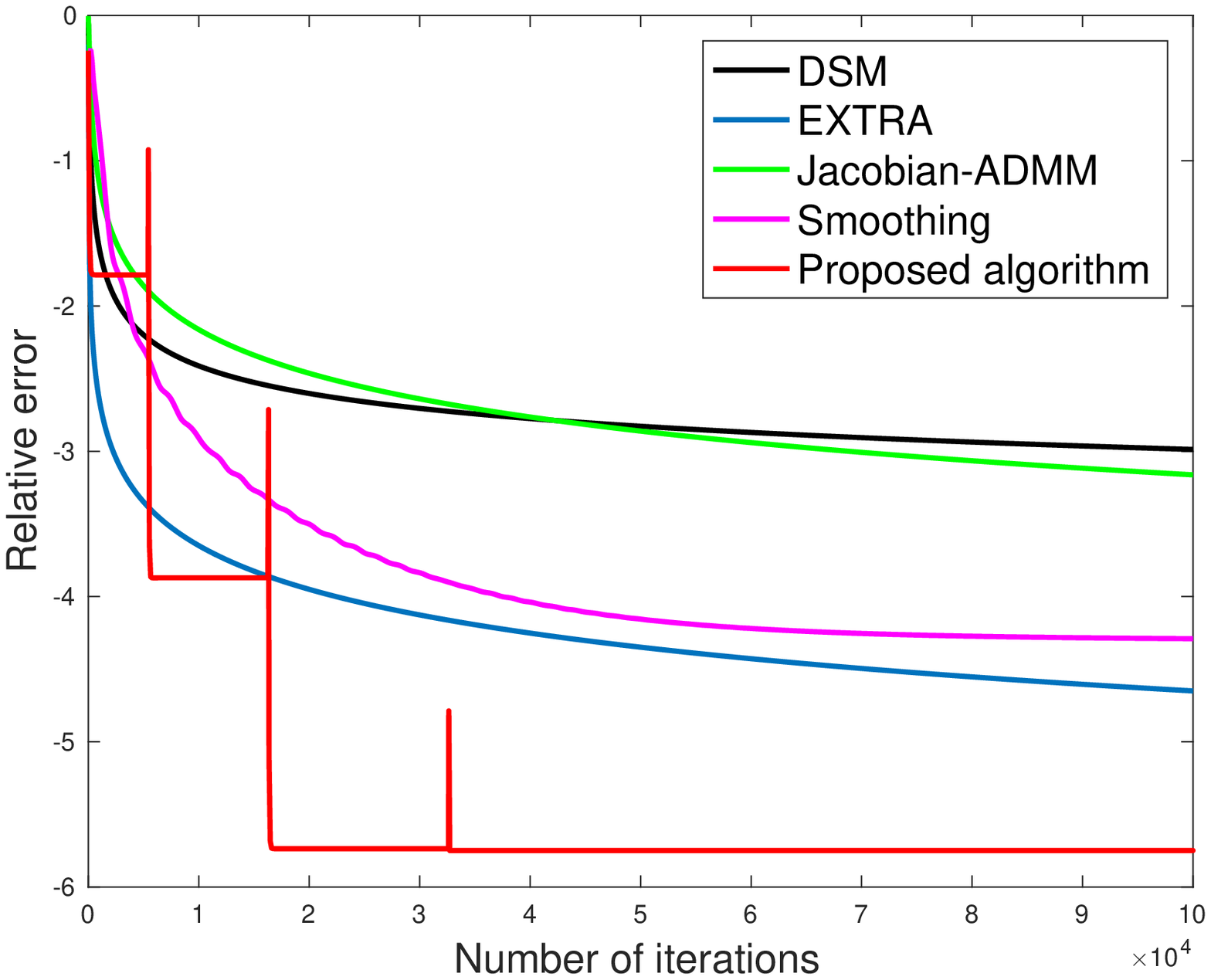}
        \vspace{-1.5em}
        \caption{$d=300$, ratio=0.1.}
    \end{subfigure}
  \caption{Performance of different algorithms under various dimensions of the vectors.}
  \label{ss2}
\end{figure*}

Finally we demonstrate the performance of our algorithm under different network connectivity ratios. In the experiment below, the number of agents is set to be $n=150$, dimension $d=100$, and all the parameters are as described above. The results are shown in Fig. \ref{ss3}.

\begin{figure*}[ht!] 
    \centering
    \begin{subfigure}[t]{0.28\textwidth}
        \centering
        \includegraphics[height=3cm] {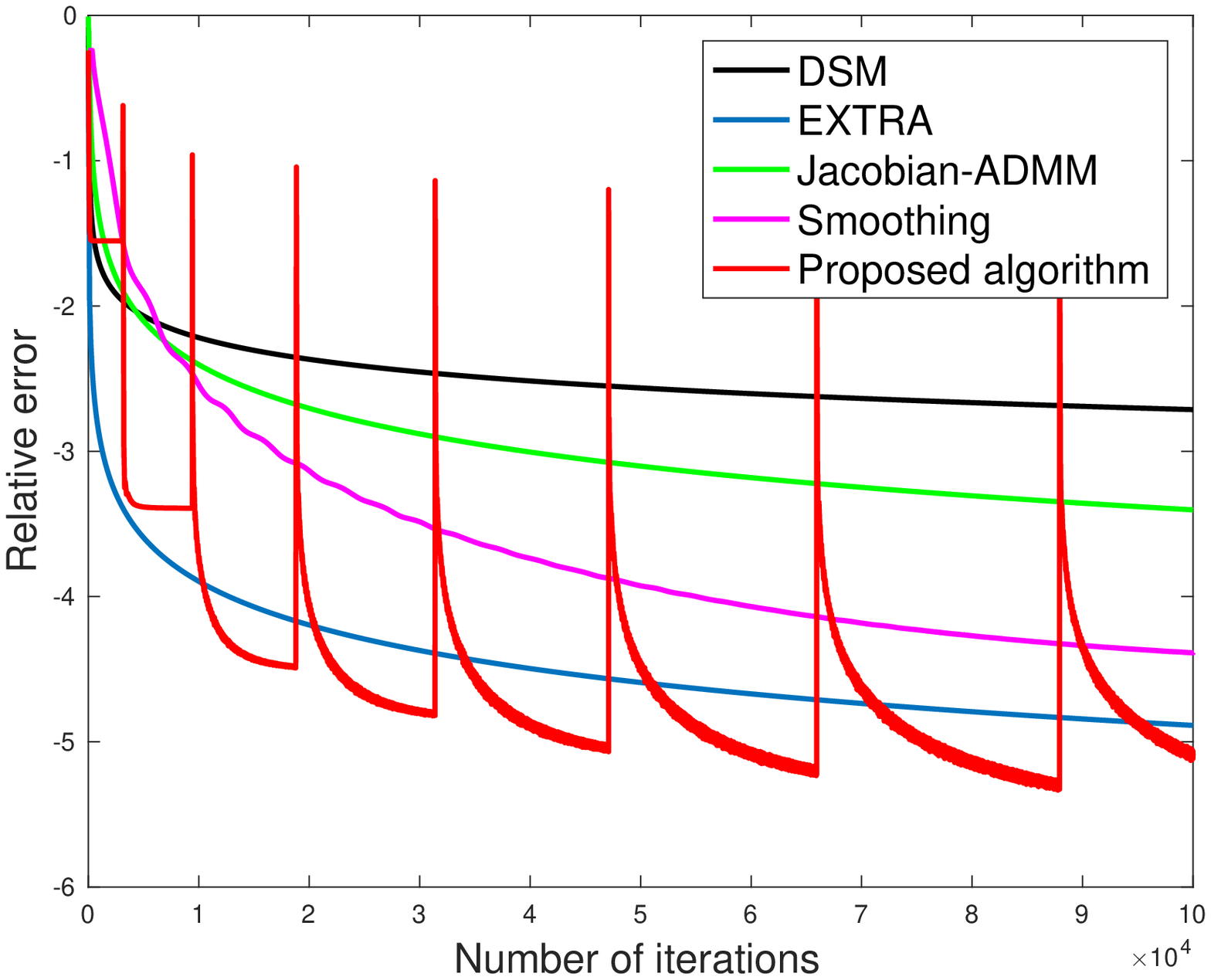}
        \vspace{-1.5em}
        \caption{$Ratio = 0.05$}
    \end{subfigure}%
    ~ 
    \begin{subfigure}[t]{0.28\textwidth}
        \centering
        \includegraphics[height=3cm] {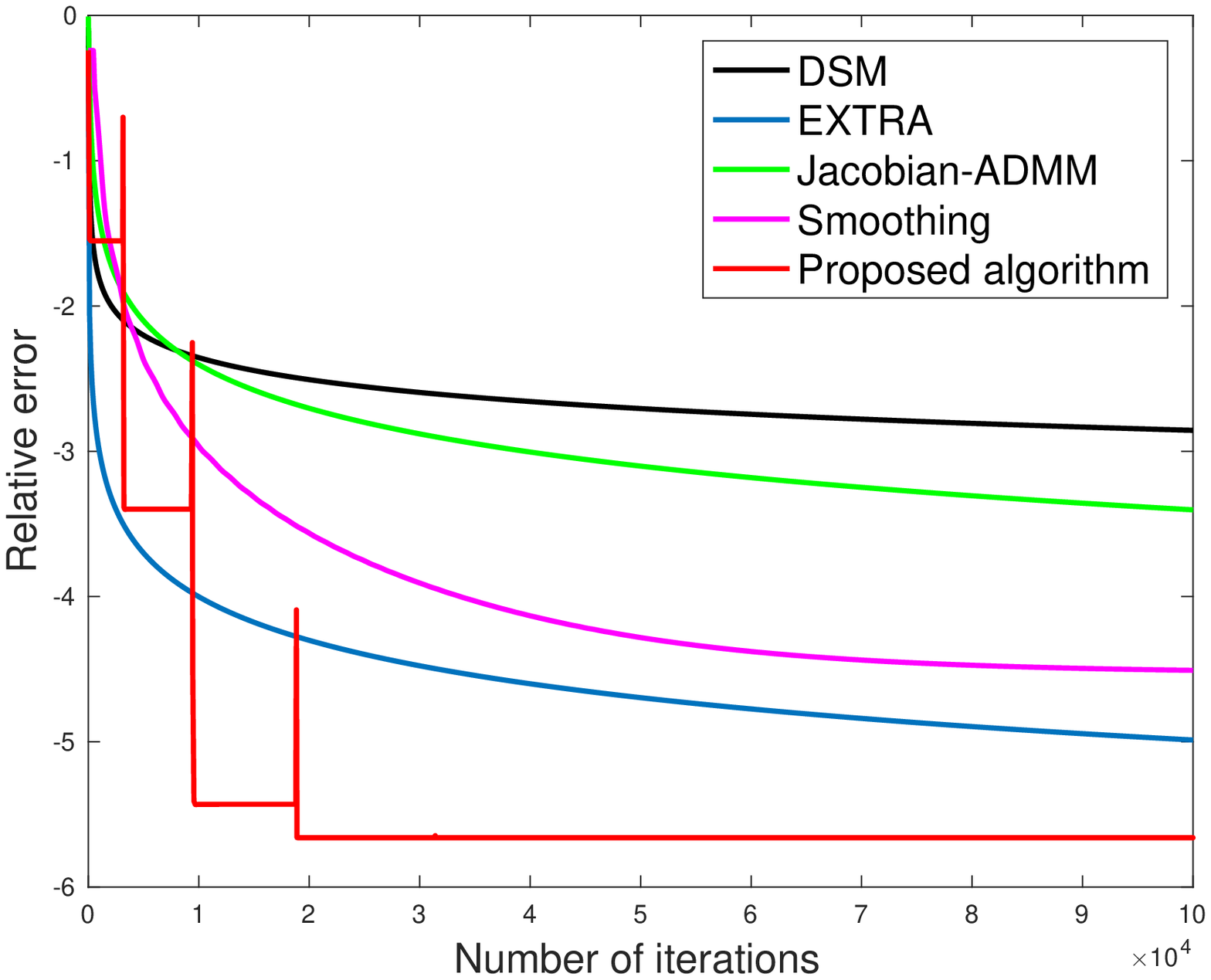}
        \vspace{-1.5em}
        \caption{$Ratio = 0.15$}
    \end{subfigure}
    ~
    \begin{subfigure}[t]{0.28\textwidth}
        \centering
        \includegraphics[height=3cm] {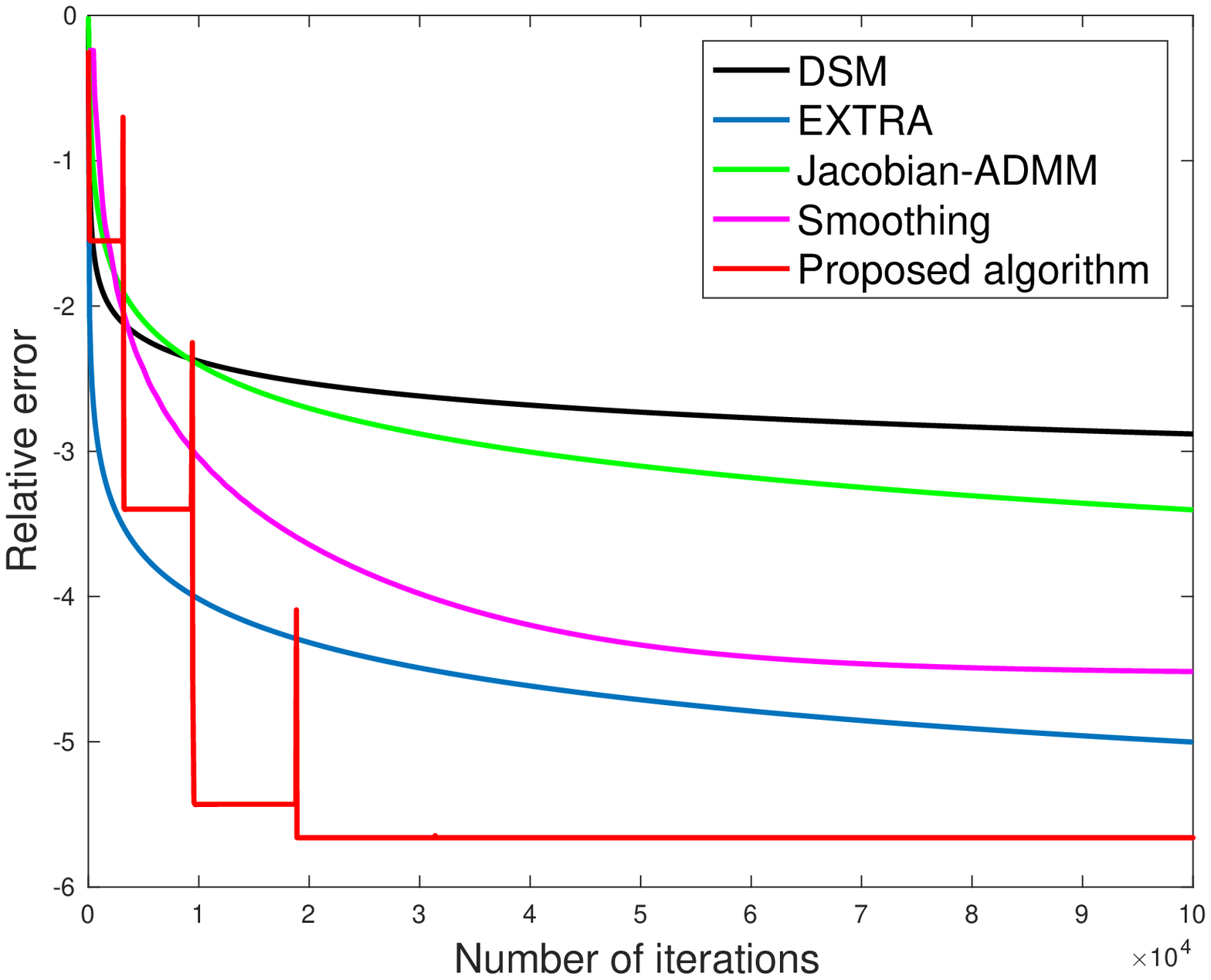}
        \vspace{-1.5em}
        \caption{$Ratio = 0.3$}
    \end{subfigure}
 
    \caption{Comparison of different algorithms on networks of different connectivity ratios.}
    \label{ss3}
\end{figure*}

\end{document}

%% file: xiaohan.tex
\newcommand{\rom}[1]{\uppercase\expandafter{\romannumeral #1\relax}}
\newcommand{\beas}{\begin{eqnarray*}}
\newcommand{\enas}{\end{eqnarray*}}
\newcommand{\bea}{\begin{eqnarray}}
\newcommand{\ena}{\end{eqnarray}}
\newcommand{\bms}{\begin{multline*}}
\newcommand{\ems}{\end{multline*}}
\newcommand{\bels}{\begin{align*}}
\newcommand{\enls}{\end{align*}}
\newcommand{\bel}{\begin{align}}
\newcommand{\enl}{\end{align}}

\newcommand{\ignore}[1]{}

\newtheorem{theorem}{Theorem}[section]
\newtheorem{corollary}{Corollary}[section]

\newtheorem{remark}{Remark}[section]
\newtheorem{lemma}{Lemma}[section]
\newtheorem{definition}{Definition}[section]

\newtheorem{assumption}{Assumption}[section]

\def\blfootnote{\xdef\@thefnmark{}\@footnotetext}

%% file: stas.tex
\def\mc{\mathcal}

\def\l{\left}
\def\r{\right}

\newcommand{\mb}{\mathbb}
\newcommand\argmin{\mathop{\mbox{argmin}}}

\newcommand{\mf}[1]{\mathbf{#1}}

\newcommand{\wh}{\widehat}

\newcommand{\dotp}[2]{\left\langle#1,#2\right\rangle}